\documentclass[11pt,twoside]{article}
\usepackage{amssymb, amsfonts, latexsym, amsthm, amsmath, verbatim, enumerate,titling, graphicx}

\newcommand{\R}{\mathbb{R}}

\newcommand{\vv}{\mathbf{v}}

\newtheorem{thm}{Theorem}[section]

\theoremstyle{plain}

\newtheorem{Proposition}[thm]{Proposition}
\newtheorem{cor}[thm]{Corollary}

\theoremstyle{definition}
\newtheorem{defn}[thm]{Definition}
\newtheorem{ex}[thm]{Example}

\newtheorem{lemma}[thm]{Lemma}

\usepackage{fancyhdr}
\usepackage[left=1in,top=1in,right=1in,bottom=1in]{geometry}

\graphicspath{{diagrams/}}

\title{A Generalized Family of Multidimensional Continued Fractions: TRIP Maps\thanks{The authors thank the National Science Foundation for their support of this research via grant DMS-0850577}}

\author{Krishna Dasaratha \thanks{K. Dasaratha-Harvard University} \and Laure Flapan\thanks{L. Flapan- Yale University} \and Thomas Garrity \thanks{T. Garrity- Williams College, tgarrity@williams.edu}
 \and Chansoo Lee\thanks{C. Lee-Williams College} \and 
Cornelia Mihaila\thanks{C. Mihaila- Wellesley College} \and Nicholas Neumann-Chun\thanks{N. Neumann-Chun- Williams College}
 \and Sarah Peluse \and Matthew Stroffregen\thanks{M. Stroffregen-University of Pittsburgh}}

\usepackage[pdfsubject={Continued Fractions},pdfkeywords={Triangle Function}]{hyperref}

\begin{document}

\maketitle

\begin{abstract}
Most well-known multidimensional continued fractions, including the M\"{o}nkemeyer map and the triangle map, are generated by repeatedly subdividing triangles. This paper constructs a family of multidimensional continued fractions by permuting the vertices of these triangles before and after each subdivision. We obtain an even larger class of multidimensional continued fractions by composing the maps in the family. These include the algorithms of Brun, Parry-Daniels and G\"{u}ting. We give criteria for when multidimensional continued fractions associate sequences to unique points, which allows us to determine when   periodicity of the corresponding multidimensional continued fraction corresponds to pairs of real numbers being cubic irrationals in the same number field.
\end{abstract}

\vskip 20pt

\section{Introduction}

In 1848, Charles Hermite \cite{HermiteC} asked Carl Jacobi for a way to represent real numbers as sequences of nonnegative integers such that a number's algebraic properties are revealed by the periodicity of its sequence. Specifically, Hermite wanted an algorithm that returns an eventually periodic sequence of integers if and only if its input is a cubic irrational. There have been many attempts to product such an algorithm. Since the continued fraction expansion of a real number is eventually periodic if and only if that number is a quadratic irrational, such attempts are known as ``multidimensional continued fractions."  Of course, continued fractions have other uses than just trying to describe quadratic irrationals, such as excellent Diophantine properties and interesting ergodic properties.  For discussion of a number of well-known multidimensional continued fractions, see \cite{SchweigerF00}, \cite{SchweigerF95}, \cite{BrentjesAJ81}. 
Most of these algorithms can be understood in terms of divisions of a triangle, so it is not far-fetched to think that there should be a way to describe them in a common language. (This is in part the goal of the work of Lagarias in \cite{Lagarias93}.)  In this paper, we develop a language by generalizing the triangle map, which was introduced in \cite{GarrityT01} and further developed in \cite{GarrityT05}, \cite{Schweiger05} and \cite{SchweigerF08}. This generalization produces a family of 216 multidimensional continued fraction algorithms. We refer to these algorithms as ``TRIP Maps," which is short for ``Triangle Permutation Maps." As we will show, these 216 TRIP Maps can be combined to produce some of the well-studied algorithms mentioned above.

We begin by introducing, in section \ref{triangle}, the motivating subdivisions of the triangle with vertices $(0,0)$, $(1,0)$, and $(1,1)$ as presented in \cite{GarrityT01},\cite{GarrityT05}. Then, in section \ref{intro216}, we construct the 216 TRIP Maps. Section \ref{examples} presents some examples of TRIP Maps and their properties with respect to periodicity. Section \ref{correspondence} expresses a number of well-known multidimensional continued fraction algorithms in terms of Combo TRIP Maps. In section \ref{periodicity}, we present results about when the periodicity of sequences generated by TRIP Maps  correspond to rationals, quadratic irrationals, or cubic irrationals. Section \ref{uniqueness} deals with the  when periodicity for a given TRIP map specifies unique points.  This is the longest and most technical part of the  paper.  In part, the argument comes down to a long case-by-case analysis of the many TRIP maps.  Section \ref{generalization} generalizes the TRIP Maps to higher dimensions by partitioning the $n$-dimensional simplex.

\section{The Triangle Division}\label{triangle}
We begin with a brief explanation of the triangle division presented in \cite{GarrityT01}. As discussed in that earlier paper, this is (one of many) generalizations of the traditional continued fraction algorithm.  Define
$$\triangle^* = \{(b_0,b_1,b_2): b_0 \geq b_1 \geq b_2 > 0 \}$$
This is a cone in the first octant of $\R^3$, but can be thought of as  a ``triangle" in $\mathbb{R}^3$.  Define the projection map $\pi: \R^3 \to \R^3$ to be
$$\pi(b_0,b_1,b_2) = \left(\frac{b_1}{b_0}, \frac{b_2}{b_0}\right)$$
Then, the image of $\triangle^*$ under $\pi$,
$$\triangle = \{ (1,x,y) : 1 \geq x \geq y > 0 \},$$
is an actual triangle in $\mathbb{R}^2$ with vertices $(0,0),$ $(1,0)$, and $(1,1)$.

Define 
$$\vv_1 = \begin{pmatrix} 1 \\ 0 \\ 0 \end{pmatrix}, 
\vv_2 = \begin{pmatrix} 1 \\ 1 \\ 0 \end{pmatrix},
\vv_3 = \begin{pmatrix} 1 \\ 1 \\ 1 \end{pmatrix}.$$
Note that $\pi$ maps $\vv_1,\vv_2,$ and $\vv_3$ to the vertices of $\triangle$. This means the matrix 
$$(\vv_1 \ \vv_2 \ \vv_3)= B=\begin{pmatrix}
1 & 1 & 1 \\ 
0 & 1 & 1 \\
0 & 0 & 1 \\
\end{pmatrix}$$\
is the change of basis matrix from the triangle coordinates to the standard basis. This will allow us to use matrices to partition $\triangle.$

Now, in order to partition $\triangle$, we consider the following two matrices:
$$A_0 = \begin{pmatrix}
0 & 0 & 1 \\ 
1 & 0 & 0 \\
0 & 1 & 1 \\
\end{pmatrix}, 
A_1 = \begin{pmatrix}
1 & 0 & 1 \\ 
0 & 1 & 0 \\
0 & 0 & 1 \\
\end{pmatrix}$$\

Note that
$$(\vv_1 \ \vv_2 \ \vv_3)A_0=(\vv_2\ \vv_3 \ \vv_1+\vv_3)$$
$$(\vv_1\   \vv_2\ \vv_3)A_1=(\vv_1\  \vv_2\  \vv_1+\vv_3)$$

So, $A_0$ and $A_1$ act on the matrix $(\vv_1\  \vv_2 \ \vv_3)$ to give a disjoint bipartition of $\triangle$:

\begin{center}
\includegraphics[scale = .6]{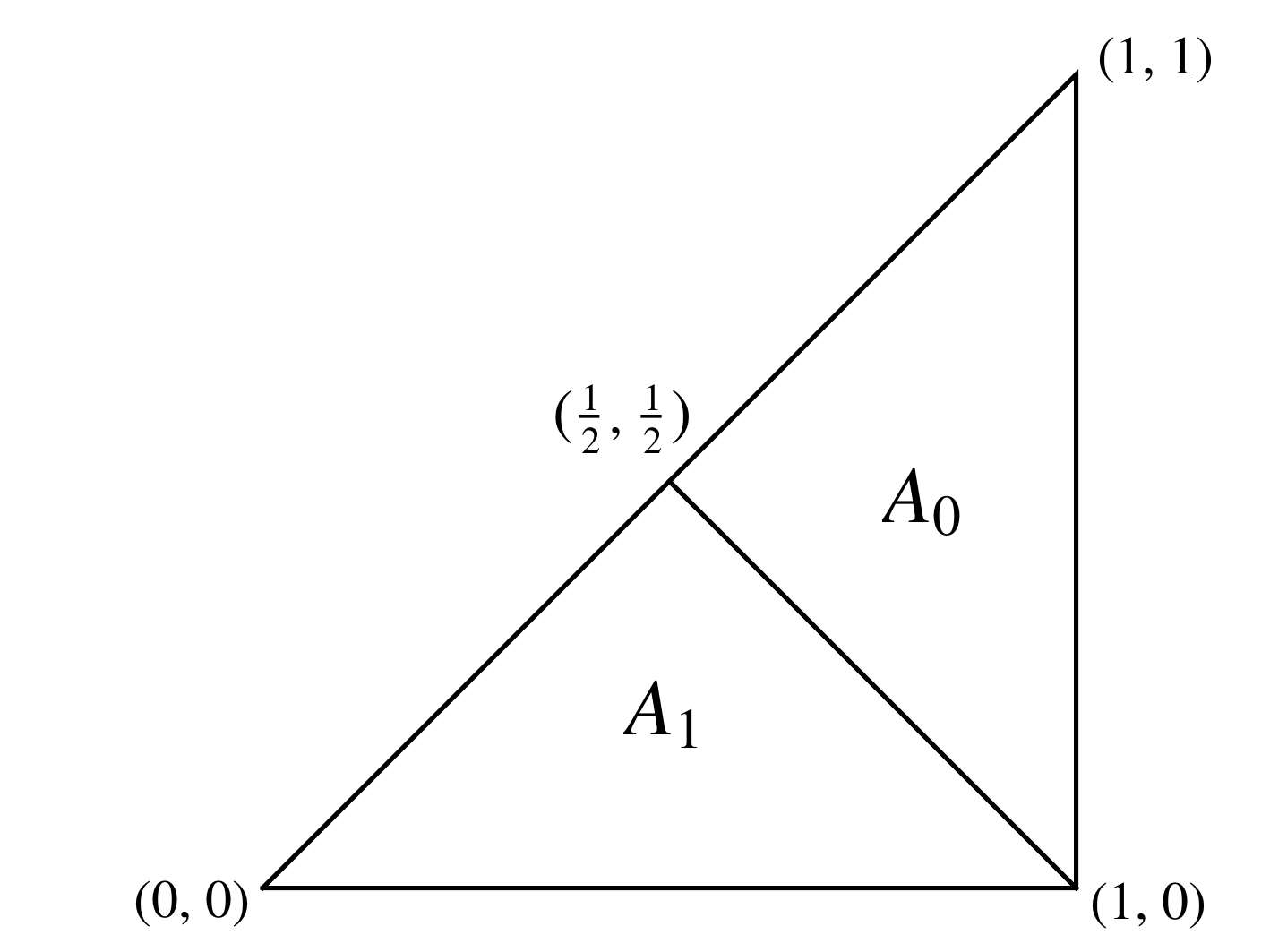}
\end{center}

This gives us the first step in the partitioning of $\triangle$.  Note that the matrices $A_0, A_1,$ and $(\vv_1 \ \vv_2 \ \vv_3)$ all have determinant 1.

Now, consider the result of applying $A_1$ a non-negative number of times, say $k$, to the original vertices of the triangle followed by applying  $A_0$ once. Then the vertices of the original triangle are mapped to the triangle 
$$\triangle_k = \{ (1,x,y) \in \triangle : 1 - x - ky \geq 0 > 1- x - (k+1)y\}$$ 
The following diagram shows the subtriangles $\triangle_0$, $\triangle_1$, and $\triangle_2.$

\begin{center}
\includegraphics[scale=.33]{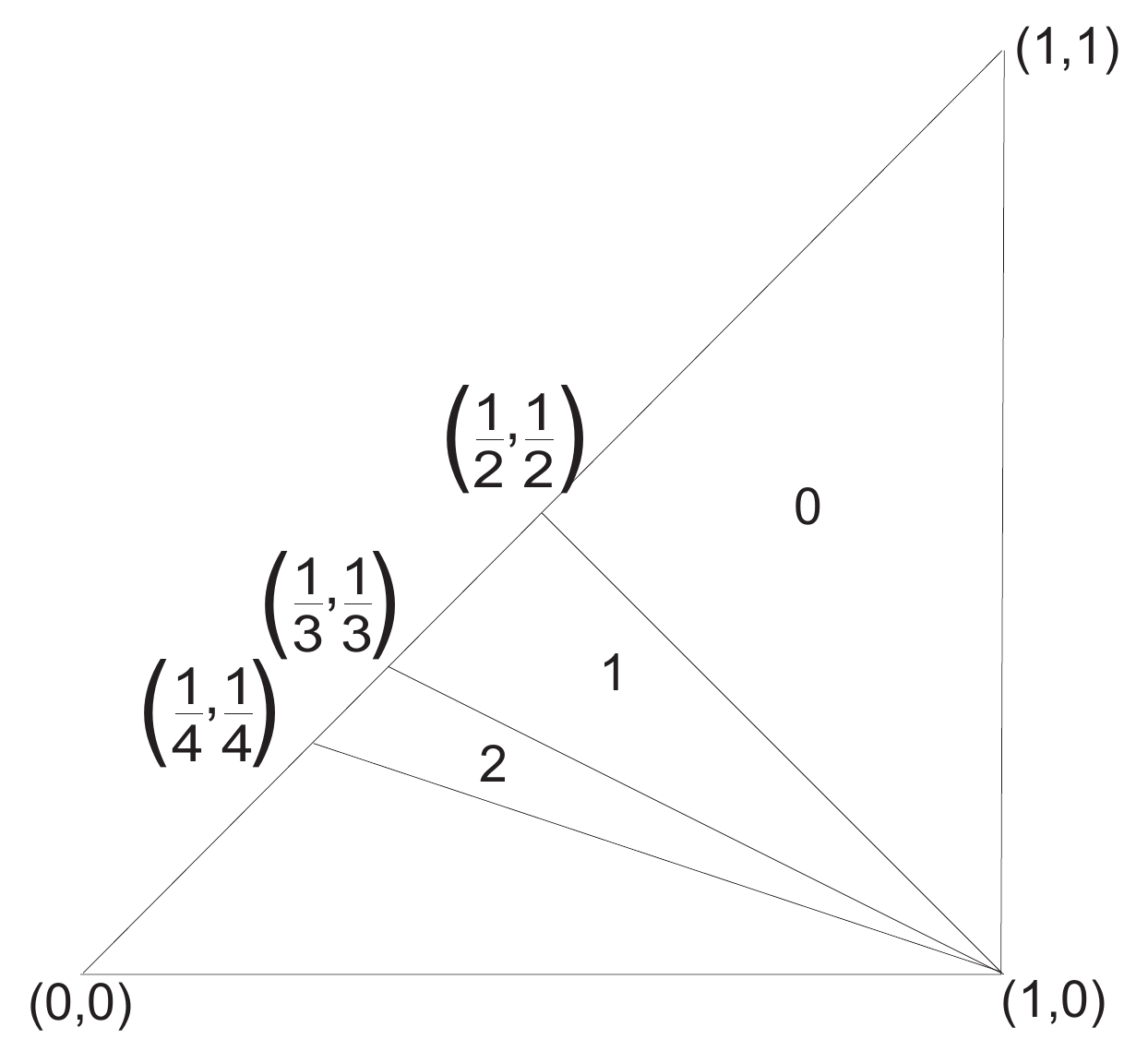}
\end{center}

We will define the triangle function $T: \triangle_k \to \triangle$ as a collection of maps, with each being a bijective map from  the subtriangle $\triangle_k$ onto $\triangle$. For any point $(x, y) \in \triangle_k$ with the standard basis, we first change the basis to triangle coordinates by multiplying by $(\vv_1 \ \vv_2 \ \vv_3)$, apply the inverse of $A_0$ and the inverse of $A_1^k$, and then finally change the basis back to the standard basis. That is,
\begin{equation}
\label{eq_tfunction}
T(x,y) = \pi( (1 \ x \ y)\left( (\vv_1 \ \vv_2 \ \vv_3)(A_0)^{-1} (A_1)^{-k} (\vv_1 \ \vv_2 \ \vv_3)^{-1}\right)^T). 
\end{equation}

This yields the following definition: $$T(x,y):=\left(\frac{y}{x},\frac{1-x-ky}{x}\right)$$ where $k=\lfloor\frac{1-x}{y}\rfloor$. This map is analogous to the Gauss map for continued fractions.  Using $T$, we can define the \emph{Triangle Sequence}, $\{a_k\}^{\infty}_{k=1}$ , of a pair in $\triangle$ by setting $a_n$ to be $k$ if $T^{(n)}(x,y)$ is in $\triangle_k$. 

\section{A Family of Multidimensional Continued Fractions}\label{intro216}
\subsection{Construction of TRIP Maps}
The triangle division, as described above, consists of partitioning the triangle with vertices $\vv_1, \vv_2,$  and $\vv_3$ into triangles with vertices $\vv_2, \vv_3,$ and $\vv_1 + \vv_3$ and $\vv_1,\vv_2,$ and $\vv_1+\vv_3$. The essential point to note is that this process assigns a particular ordering of vertices to both the vertices of the original triangle and the vertices of the two triangles produced. However, there is nothing canonical about these orderings. We can permute the vertices at several stages of the triangle division. By considering all possible permutations we generate a family of 216 maps, each corresponding to a partition of $\triangle$.

Specifically, we allow a permutation of the vertices of the initial triangle as well as a permutation of the vertices of the triangles obtained after applying $A_0$ and $A_1$. First, we permute the vertices by $\sigma \in S_3$ before applying either $A_0$ or $A_1$. Once we apply either $A_0$ or $A_1$, we then permute by either $\tau_0 \in S_3$ or $\tau_1 \in S_3$, respectively. This leads to the following definition:

\begin{defn}For every $(\sigma, \tau_0, \tau_1) \in S^3_3$, define
$$F_0 = \sigma A_0 \tau_0 \text{ and } F_1 = \sigma A_1 \tau_1$$
by thinking of $\sigma,$ $\tau_0$, and $\tau_1$ as column permutation matrices.
\end{defn}

Note that applying $F_0$ and $F_1$ partitions any triangle into two subtriangles. Thus, given some $(\sigma, \tau_0, \tau_1) \in S^3_3$, we can partition the triangle $\triangle$ using the matrices $F_0$ and $F_1$ instead of $A_0$ and $A_1$. This produces a map that is similar to, but \emph{not} the same as, the triangle map described in section \ref{triangle}. We call each of these maps a Triangle Partition Map, or ``TRIP Map" for short. Because $\left|S_3^3\right| = 216$, the family of TRIP Maps has 216 elements.

\subsection{TRIP Tree Sequence}
There are two reasonable ways to construct an integer sequence using a TRIP Map. One possibility is to keep track of which of the two subtriangles defined by $A_0$ and $A_1$ a given point is in at each step of the triangle division. We formalize this as follows.

Let $\triangle(i_0,i_1,\ldots ,i_n)$ be the triangle with vertices given by the columns of $$(\vv_1 \ \vv_2 \ \vv_3) F_{i_0}F_{i_1}F_{i_2}\cdots F_{i_n}.$$

\begin{defn}
Given $(x, y) \in \triangle$, inductively define $i_n \in \{0, 1\}$ as the unique element such that $(x, y) \in \triangle(i_0,i_1,\ldots ,i_n).$ The sequence $(i_0,i_1,\ldots )$ is the \textbf{TRIP tree sequence} of $(x, y)$ with respect to $(\sigma, \tau_0, \tau_1)$.
\end{defn}

For example, if $(\alpha, \beta)$ is in the subtriangle obtained by applying $F_0$, $F_1$, and then $F_1$ in succession, then the first three terms of the triangle tree sequence of $(\alpha,\beta)$ are $(0, 1, 1,\ldots )$.

\subsection{TRIP Sequence}

The second method is essentially the same as the definition of the triangle sequence given in the section two. To begin, recall that we had a notion of ``subtriangle $\triangle_k$" in the original triangle map. We extend this as: given $(\sigma, \tau_0, \tau_1)$, let $\triangle_k$ be the image of the triangle $\triangle$ under $F_1^kF_0$. We now define functions $T: \triangle \rightarrow \triangle$ mapping each subtriangle $\triangle_k$ bijectively to $\triangle$.

\begin{defn}
The \textbf{triangle function} $T_{\sigma, \tau_0,\tau_1}$ is given by
$$T_{\sigma, \tau_0, \tau_1}(x, y) =\pi((1, x, y) ( (\vv_1\ \ \vv_2\ \ \vv_3) F_0^{-1} F_1^{-k}  (\vv_1\ \ \vv_2\ \ \vv_3)^{-1})^T)  \text{ when } (x,y)\in\triangle_k$$
\end{defn}

(In much of the rest of the paper we will be setting $(\vv_1\ \ \vv_2\ \ \vv_3) =B$, which means we can write  
$T_{\sigma, \tau_0, \tau_1}(x, y) =\pi((1, x, y) ( B F_0^{-1} F_1^{-k}  (\vv_1\ \ \vv_2\ \ \vv_3)^{-1})^T$.)

Note that $(\vv_1\ \ \vv_2\ \ \vv_3)$ is just the change of basis matrix from the basis $\{\vv_1,\vv_2,\vv_3\}$ to the standard basis. Additionally note that we take the transpose of the matrix $ (\vv_1\ \ \vv_2\ \ \vv_3) F_0^{-1} F_1^{-k}  (\vv_1\ \ \vv_2\ \ \vv_3)^{-1}$ because we have written our matrices with vertices as columns but we are multiplying by the row vector $(1,x,y)$. 

\begin{defn}
Let $F_0$ and $F_1$ be generated from some triplet of permutations. We define $\Delta_n$ to be the triangle with vertices given by the columns of $(\vv_1\ \ \vv_2\ \ \vv_3){F_1}^n{F_0}$.
\end{defn}

We now use this to define the second method of constructing an integer sequence.

\begin{defn}
Given an $(x, y)\in\triangle$, define $a_n$ to be the non-negative integer such that $\big[T_{\sigma, \tau_0, \tau_1}\big]^n(x, y)$ is in $\triangle_{a_n}$. The \textbf{TRIP sequence} of $(x, y)$ with respect to $(\sigma, \tau_0, \tau_1)$ is $(a_0,a_1,\ldots )$.
\end{defn}

The two methods described above correspond in the the following way. The triangle tree sequence begins with the term one repeated $a_0$ times followed by a zero. The next terms are $a_1$ ones followed by a zero. The next terms are $a_2$ ones followed by a zero, and so on.

\subsection{An Even Larger Family: Combo TRIP Maps}

Finally, and most importantly, we can obtain a much larger family of functions by composing TRIP Maps. For example, we can carry out the first subdivision of $\triangle$ using $(\sigma, \tau_0, \tau_1)$, the second subdivision using $(\sigma', \tau_0', \tau_1')$, and so forth. As we will see in Section \ref{correspondence}, many known multidimensional continued fractions are such combinations of the $216$ Trip Maps. We call any  combination of our maps {\it Combo TRIP maps}.

\section{Some TRIP maps}\label{examples}
While some of the Trip Maps are quite simple to write down, the vast majority are quite complicated.  Below, we present a few of the simple maps in our family.

\begin{ex}{The Triangle Map}

The original triangle map corresponds to the Trip Map generated by three identity permutations $(e, e, e)$. Thus,
\[T_{(e,e,e)}(1,x,y)=(1, x, y) \cdot (B A_0^{-1} A_1^{-k} B^{-1})^T.\]
Recall that we can easily calculate
\[( (\vv_1\ \ \vv_2\ \ \vv_3) A_0^{-1} A_1^{-k}  (\vv_1\ \ \vv_2\ \ \vv_3)^{-1})^T = \left(
\begin{array}{ccc}
 0 & 0 & 1 \\
 1 & 0 & -1 \\
 0 & 1 & -k
\end{array}
\right)\]
(It is generally much harder, though always possible, to find a closed form for $( (\vv_1\ \ \vv_2\ \ \vv_3) F_0^{-1} F_1^{-k}$  $(\vv_1\ \ \vv_2\ \ \vv_3)^{-1})^T$.)  This yields
$$T_{(e,e,e)}(x,y)=\left(\frac{y}{x}, \frac{1 - x - ky}{x}\right),$$ when $(x,y)\in \triangle_k$. 
Note that this definition of the triangle map agrees with the definition in \cite{GarrityT01}.

Consider the point $(\frac{1}{\sqrt[3]{2}}, \frac{1}{\sqrt[3]{4}})$. We will compute the first few terms of both its triangle tree sequence and trip sequence with respect to $(e, e, e)$. 

\textbf{Triangle Tree Sequence of $(\frac{1}{\sqrt[3]{2}}, \frac{1}{\sqrt[3]{4}})$}

We begin by applying $A_0$ and $A_1$ to the matrix $(\vv_1 \ \vv_2 \ \vv_3)$ where $\vv_1$,$\vv_2$, and $\vv_3$ are the vertices of $\triangle$.  Then $A_0$ yields a triangle $\triangle(0)$ with vertices $\vv_2,\vv_3,$ and $\vv_1+\vv_3$ and $A_1$ yields a triangle $\triangle(1)$ with vertices $\vv_1,\vv_2,$ and $\vv_1+\vv_3$.  These triangles correspond to the triangles in $\mathbb{R}^2$ with vertices $(1,0), (1,1), $ and $ (\frac{1}{2},\frac{1}{2})$ and $(0,0), (1,0), $ and $ (\frac{1}{2},\frac{1}{2})$, respectively.  Since the point $(\frac{1}{\sqrt[3]{2}}, \frac{1}{\sqrt[3]{4}})$ is located inside of $\triangle(0)$, the first term of its triangle tree sequence is $0$.  

Now apply $A_0$ and $A_1$ to the matrix $(\vv_2 \ \vv_3 \ \vv_1+\vv_3)$, whose columns are the vertices of $\triangle(0)$.  Then $A_0$ yields a triangle $\triangle(0,0)$ with vertices $\vv_3,\vv_1+\vv_3,$ and $\vv_1+\vv_2+\vv_3$ and $A_1$ yields a triangle $\triangle(0,1)$ with vertices $\vv_2,\vv_3,$ and $\vv_1+\vv_2+\vv_3$.  These triangles correspond to the triangles in $\mathbb{R}^2$ with vertices $(1,1), (\frac{1}{2},\frac{1}{2}), $ and $ (\frac{2}{3},\frac{1}{3})$ and $(1,0), (1,1), $ and $ (\frac{2}{3},\frac{1}{3})$, respectively.  Since the point lies in $\triangle(0,0)$, the second term of its triangle tree sequence is $0$. 

To determine the third term in the triangle tree sequence, apply $A_0$ and $A_1$ to the matrix
\[(\vv_3 \ \vv_1+\vv_3 \ \vv_1+\vv_2+\vv_3),\]
whose columns are the vertices of $\triangle(0,0)$.  The matrix $A_0$ yields $\triangle(0,0,0)$ with vertices
\[(\vv_1+\vv_3,\vv_1+\vv_2+\vv_3,\vv_1+\vv_2+2\vv_3)\]
and $A_1$ yields $\triangle(0,0,1)$ with vertices
\[(\vv_3,\vv_1+\vv_3,\vv_1+\vv_2+2\vv_3).\]
These triangles correspond to the triangles in $\mathbb{R}^2$ with vertices $(\frac{1}{2},\frac{1}{2}), (\frac{2}{3},\frac{1}{3}), $ and $ (\frac{3}{4},\frac{2}{4})$ and $(1,1), (\frac{1}{2},\frac{1}{2}), $ and $ (\frac{3}{4},\frac{2}{4})$, respectively.  Since the point $(\frac{1}{\sqrt[3]{2}}, \frac{1}{\sqrt[3]{4}})$ lies in $\triangle(0,0,1)$, the third term of its triangle tree sequence is $1$.

Writing out the first three terms, we get that the triangle tree sequence of $(\frac{1}{\sqrt[3]{2}}, \frac{1}{\sqrt[3]{4}})$ is $(0,0,1,\ldots)$.

\textbf{TRIP Sequence of $(\frac{1}{\sqrt[3]{2}}, \frac{1}{\sqrt[3]{4}})$}

We know that $(\frac{1}{\sqrt[3]{2}}, \frac{1}{\sqrt[3]{4}})$ lies in
$$\triangle_0 = \{ (1,x,y) \in \triangle : 1 - x  \geq 0 > 1- x - y\}$$
so the first term of the TRIP sequence of $(\frac{1}{\sqrt[3]{2}}, \frac{1}{\sqrt[3]{4}})$ is $0$.

Now apply the map $T_{(e,e,e)}=\left(\frac{y}{x}, \frac{1 - x}{x}\right)$ to the point $(\frac{1}{\sqrt[3]{2}}, \frac{1}{\sqrt[3]{4}})$.
$$T_{e,e,e} \Big(\frac{1}{\sqrt[3]{2}}, \frac{1}{\sqrt[3]{4}}\Big)=\Big(\frac{1}{\sqrt[3]{2}},\sqrt[3]{2}-1\Big)$$

This point lies in $\triangle_0$, so the second term of the TRIP sequence is $0$.  

We now apply $T_{(e,e,e)}$ to the point $(\frac{1}{\sqrt[3]{2}},\sqrt[3]{2}-1)$ and we get 
$$T_{(e,e,e)}\Big(\frac{1}{\sqrt[3]{2}},\sqrt[3]{2}-1\Big)=(\sqrt[3]{4}-\sqrt[3]{2},\sqrt[3]{2}-1).$$
This point lies in $\triangle_2= \{ (1,x,y) \in \triangle : 1 - x - 2y \geq 0 > 1- x - 3y\}$, so the third term of the Trip sequence is $2$.

Writing out the first three terms, we get that the Trip sequence of $(\frac{1}{\sqrt[3]{2}}, \frac{1}{\sqrt[3]{4}})$ is $(0,0,2,\ldots)$.

\end{ex}

\begin{ex}{The M\"{o}nkemeyer Map}

We will see that the M\"{o}nkemeyer map, described in \cite{SchweigerF00} and \cite{PantiG00}, is the TRIP Map generated by the permutations $(e, (1\ 3\ 2), (2\ 3))$ (see Section \ref{correspondence}.)  It corresponds to the division of the triangle shown below.

\begin{center}
\includegraphics[scale=.5, trim = 0mm 0mm 0mm 0mm, clip]{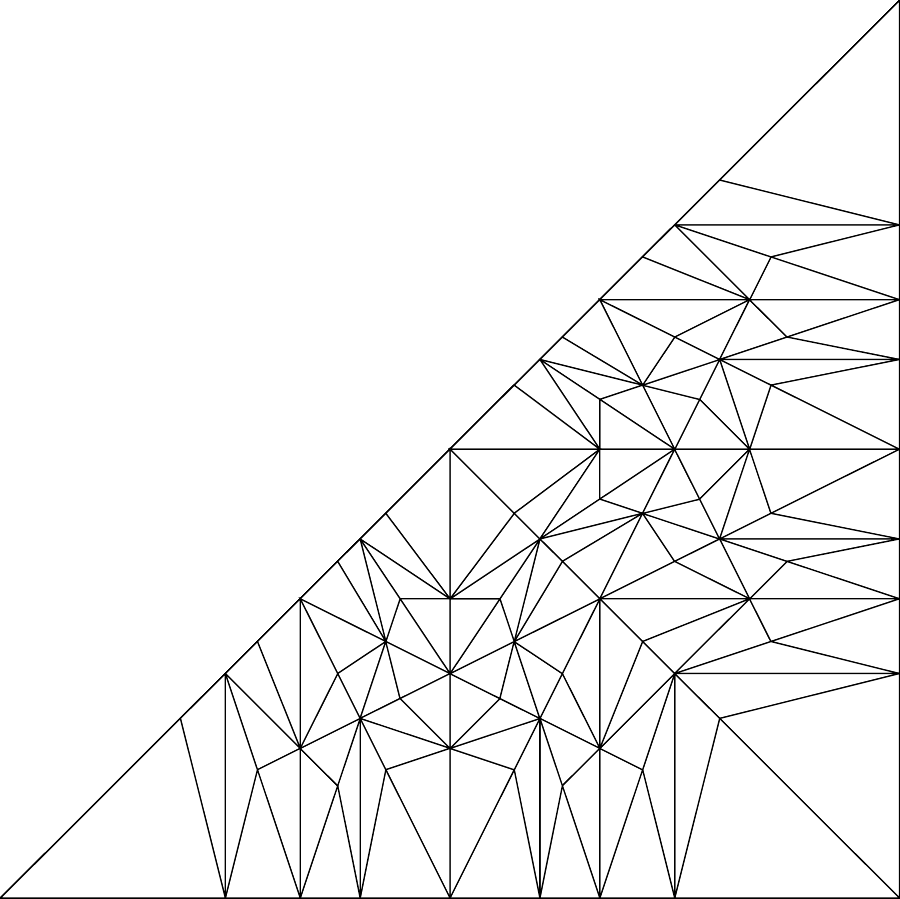}
\end{center}

The M\"{o}nkemeyer map is defined on $\triangle$ in the following manner:
\[
T(x,y)=\begin{cases}
(\frac{x}{1-y},\frac{x-y}{1-y}) & \mbox{if \ensuremath{x+y\leq1}}\\
(\frac{1-y}{x},\frac{x-y}{x}) & \mbox{if \ensuremath{x+y\geq1}}
\end{cases}
\]

The set $\{(x,y)|x+y\leq1\}\cap\triangle$ is the triangle with vertices
$(1,0),$ $(1,1),$ and $(\frac{1}{2},\frac{1}{2}),$ which will be
denoted $\triangle_{1}.$ The set $\{(x,y)|x+y\geq1\}\cap\Delta$ is
the triangle with vertices $(0,0),$ $(1,0),$ and $(\frac{1}{2},\frac{1}{2}),$
which will be denoted $\triangle_{0}.$ Consider the TRIP Map generated
by the permutations $e,$ $(1\mbox{ }3\mbox{ }2),$ and $(2\mbox{ }3).$
Note that
\begin{equation}
(B\left(\begin{array}{ccc}
1 & 0 & 0\\
0 & 1 & 0\\
0 & 0 & 1
\end{array}\right)A_{0}\left(\begin{array}{ccc}
0 & 0 & 1\\
1 & 0 & 0\\
0 & 1 & 0
\end{array}\right))^{T}=\left(\begin{array}{ccc}
0 & 1 & 0\\
0 & 1 & 1\\
1 & 2 & 1
\end{array}\right)
\end{equation}

and
\begin{equation}
(B\left(\begin{array}{ccc}
1 & 0 & 0\\
0 & 1 & 0\\
0 & 0 & 1
\end{array}\right)A_{1}\left(\begin{array}{ccc}
1 & 0 & 0\\
0 & 0 & 1\\
0 & 1 & 0
\end{array}\right))^{T}=\left(\begin{array}{ccc}
1 & 0 & 1\\
1 & 1 & 1\\
1 & 1 & 2
\end{array}\right),
\end{equation}

which is the same initial division of $\triangle$ as the M\"{o}nkemeyer
map.

On $\triangle_{0},$ we thus have
\[
((BA_{0}\left(\begin{array}{ccc}
0 & 0 & 1\\
1 & 0 & 0\\
0 & 1 & 0
\end{array}\right))^{T})^{-1}B^T=\left(\begin{array}{ccc}
0 & 1 & 0\\
1 & 0 & 1\\
0 & -1 & -1
\end{array}\right),
\]

so
\[
(x,y)\mapsto\left(\frac{1-y}{x},\frac{x-y}{x}\right),
\]

which agrees with the definition of the M\"{o}nkemeyer map above.

On $\triangle_{1},$ we thus have 

\[
((BA_{1}\left(\begin{array}{ccc}
1 & 0 & 0\\
0 & 0 & 1\\
0 & 1 & 0
\end{array}\right))^{T})^{-1}B^T=\left(\begin{array}{ccc}
1 & 0 & 0\\
0 & 1 & 1\\
-1 & 0 & -1
\end{array}\right),
\]

so
\[
(x,y)\mapsto \left(\frac{x}{1-y},\frac{1-y}{1-y}\right),
\]

which agrees with the definition of the M\"{o}nkemeyer map above. Thus,
we can conclude that the M\"{o}nkemeyer map is equivalent to the TRIP
map $T_{(e,(1\mbox{ }3\mbox{ }2),(2\mbox{ }3))}.$

\end{ex}

Both of these TRIP maps are well-known.  Our next example of a TRIP map seems to be a new multi-dimensional continued fraction, showing that our family is easily capable of producing many new multi-dimensional continued fractions.

\begin{ex}{The TRIP Map $T_{(e,e,(1\ 2))}$}

As a final example, the triangle division generated by the permutations $(e,e,(1\ 2))$ looks like
\begin{center}
\includegraphics[scale=.5]{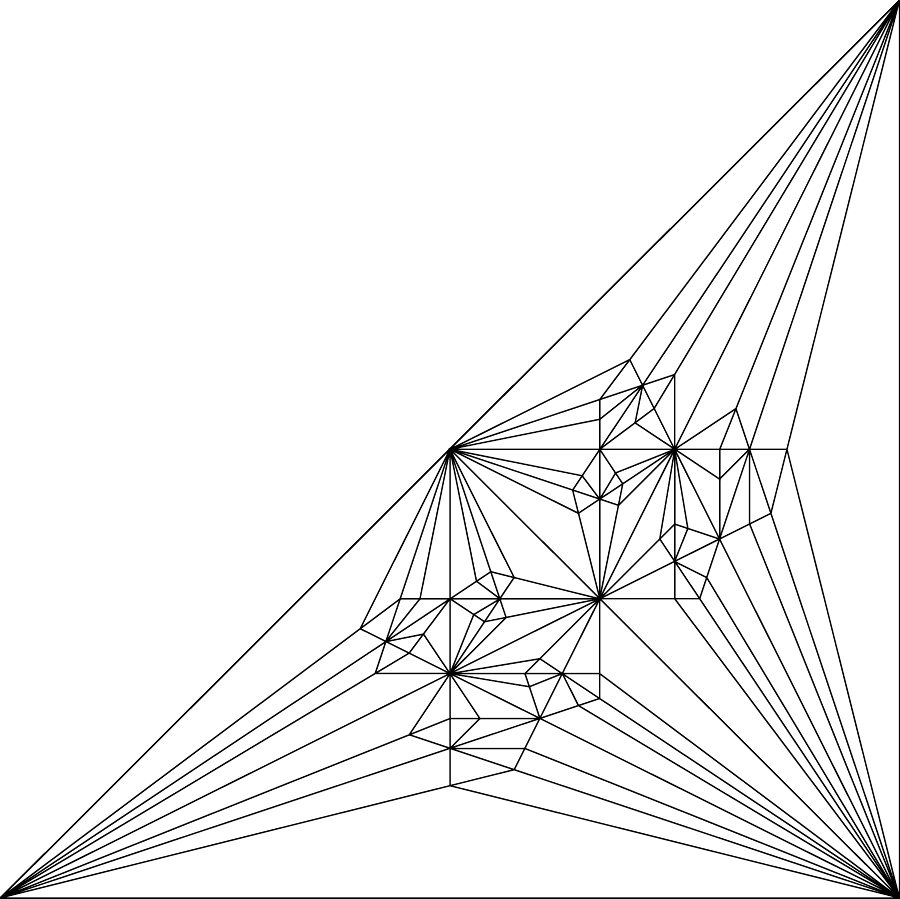}
\end{center}

The corresponding TRIP Map is 
$$T_{(e,e,(1\ 2))}=\left(\frac{(-1)^k4y}{4x-y-2+(-1)^k(2+y-2ky)},\frac{2+y-4x-(-1)^k(y-2+2ky)}{4x-y-2+(-1)^k(2+y-2ky)}\right)$$
  
  \end{ex}
  
  \section{Examples of Combo TRIP Maps} \label{correspondence}
  
  In this section, we show that several already well-known multidimensional continued fraction algorithms are Combo TRIP Maps. There are, though,  far more Combo TRIP maps than those that have already been studied.  (For background on these maps, see Schweiger \cite{SchweigerF00} .)

\begin{ex}{The Brun Algorithm}

We will show that the Brun map is  a combination of the TRIP
Map generated by the permutations $(e,e,e)$ and the TRIP Map generated by the permutations $((2\mbox{ }3),(1\mbox{ }3\mbox{ }2),(2\mbox{ }3)).$
Described in terms of TRIP Maps, the Brun algorithm first applies
$T_{(e,e,e)}$ and then applies $T_{((2\mbox{ }3),(1\mbox{ }3\mbox{ }2),(2\mbox{ }3))}$
if $(x,y)$ is not in $\triangle_{0}.$

Recall that the Brun map is defined on $\triangle$ as
\[
T(x,y)=\begin{cases}
(\frac{x}{1-x},\frac{y}{1-x}) & \mbox{if }x+y\leq1\mbox{ and }x\leq\frac{1}{2}\\
(\frac{1-x}{x},\frac{y}{x}) & \mbox{if }x+y\leq1\mbox{ and }x\geq\frac{1}{2}\\
(\frac{y}{x},\frac{1-x}{x}) & \mbox{if }x+y\geq1
\end{cases}
\]

The set $\{(x,y)|x+y\leq1\mbox{ and }x\leq\frac{1}{2}\}\cap\triangle$
is the triangle with vertices $(0,0),$ $(\frac{1}{2},0),$ and $(\frac{1}{2},\frac{1}{2}),$
which will be denoted $\triangle_{2}.$ The set $\{(x,y)|x+y\leq1\mbox{ and }x\geq\frac{1}{2}\}\cap\triangle$
is the triangle with vertices $(1,0),$ $(\frac{1}{2},0),$ and $(\frac{1}{2},\frac{1}{2}),$
which will be denoted $\triangle_{1}.$ The set $\{(x,y)|x+y\geq1\}\cap\triangle$
is the triangle with vertices $(1,0),$ $(1,1),$ and $(\frac{1}{2},\frac{1}{2}),$
which will be denoted $\triangle_{0}.$ Consider the TRIP map generated
by the permutations $(e,e,e)$ and the TRIP map generated by the permutations
$((2\mbox{ }3),(1\mbox{ }3\mbox{ }2),(2\mbox{ }3)).$ Note that 
\begin{equation}
\left[B\left(\begin{array}{ccc}
1 & 0 & 0\\
0 & 1 & 0\\
0 & 0 & 1
\end{array}\right)A_{0}\left(\begin{array}{ccc}
1 & 0 & 0\\
0 & 1 & 0\\
0 & 0 & 1
\end{array}\right)\right]^{T}=\left(\begin{array}{ccc}
1 & 1 & 0\\
1 & 1 & 1\\
2 & 1 & 1
\end{array}\right)
\end{equation}

and
\begin{equation}
\left[B\left(\begin{array}{ccc}
1 & 0 & 0\\
0 & 1 & 0\\
0 & 0 & 1
\end{array}\right)A_{1}\left(\begin{array}{ccc}
1 & 0 & 0\\
0 & 1 & 0\\
0 & 0 & 1
\end{array}\right)\right]^{T}=\left(\begin{array}{ccc}
1 & 0 & 0\\
1 & 1 & 0\\
2 & 1 & 1
\end{array}\right).
\end{equation}

Applying $T_{((2\mbox{ }3),(1\mbox{ }3\mbox{ }2),(2\mbox{ }3))}$
to the second matrix yields
\begin{equation}
\left[\left(\begin{array}{ccc}
1 & 0 & 0\\
1 & 1 & 0\\
2 & 1 & 1
\end{array}\right)^{T}\left(\begin{array}{ccc}
1 & 0 & 0\\
0 & 0 & 1\\
0 & 1 & 0
\end{array}\right)A_{0}\left(\begin{array}{ccc}
0 & 0 & 1\\
1 & 0 & 0\\
0 & 1 & 0
\end{array}\right)\right]^{T}=\left(\begin{array}{ccc}
1 & 0 & 0\\
2 & 1 & 0\\
2 & 1 & 1
\end{array}\right)
\end{equation}

and
\begin{equation}
\left[\left(\begin{array}{ccc}
1 & 0 & 0\\
1 & 1 & 0\\
2 & 1 & 1
\end{array}\right)^{T}\left(\begin{array}{ccc}
1 & 0 & 0\\
0 & 0 & 1\\
0 & 1 & 0
\end{array}\right)A_{1}\left(\begin{array}{ccc}
1 & 0 & 0\\
0 & 0 & 1\\
0 & 1 & 0
\end{array}\right)\right]^{T}=\left(\begin{array}{ccc}
1 & 1 & 0\\
2 & 1 & 0\\
2 & 1 & 1
\end{array}\right),
\end{equation}

so (1), (3), and (4) describe the vertices of $\triangle_{0},$ $\triangle_{1}$,
and $\triangle_{2},$ respectively, which is the same division as the
Brun map.

On $\triangle_{0},$ we thus have
\[
((BA_{0})^{T})^{-1}B^T=\left(\begin{array}{ccc}
0 & 0 & 1\\
1 & 0 & -1\\
0 & 1 & 0
\end{array}\right),
\]

so
\[
(x,y)\mapsto\left(\frac{y}{1-x},\frac{1-x}{1-x}\right),
\]

which agrees with the definition of the Brun map above.

On $\triangle_{1},$ we thus have
\[
\left(\left[BA_{1}\left(\begin{array}{ccc}
1 & 0 & 0\\
0 & 0 & 1\\
0 & 1 & 0
\end{array}\right)A_{0}\left(\begin{array}{ccc}
0 & 0 & 1\\
1 & 0 & 0\\
0 & 1 & 0
\end{array}\right)\right]^{T}\right)^{-1}B^T=\left(\begin{array}{ccc}
0 & 1 & 0\\
1 & -1 & 0\\
0 & 0 & 1
\end{array}\right),
\]

so
\[
(x,y)\mapsto\left(\frac{1-x}{x},\frac{y}{x}\right),
\]

which agrees with the definition of the Brun map above.

On $\triangle_{2},$ we thus have
\[
\left(\left[BA_{1}\left(\begin{array}{ccc}
1 & 0 & 0\\
0 & 0 & 1\\
0 & 1 & 0
\end{array}\right)A_{1}\left(\begin{array}{ccc}
1 & 0 & 0\\
0 & 0 & 1\\
0 & 1 & 0
\end{array}\right)\right]^{T}\right)^{-1}B^T=\left(\begin{array}{ccc}
1 & 0 & 0\\
-1 & 1 & 0\\
0 & 0 & 1
\end{array}\right),
\]

so
\[
(x,y)\mapsto\left(\frac{x}{1-x},\frac{y}{1-x}\right),
\]

which agrees with the definition of the Brun map above. Thus, we can indeed
conclude that the Brun map is the Combo TRIP
Map generated by the permutations $(e,e,e)$ and the TRIP Map generated by the permutations $((2\mbox{ }3),(1\mbox{ }3\mbox{ }2),(2\mbox{ }3)),$
where the Brun algorithm first applies
$T_{(e,e,e)}$ and then applies $T_{((2\mbox{ }3),(1\mbox{ }3\mbox{ }2),(2\mbox{ }3))}$
if $(x,y)$ is not in $\triangle_{0}.$

\end{ex}

  \begin{ex}{The Fully Subtractive Algorithm}

The fully subtractive map is defined on $\triangle$ in the following
manner:
\[
T(x,y)=\begin{cases}
(\frac{1-y}{y},\frac{x-y}{y}) & \mbox{if }y\geq\frac{1}{2}\\
(\frac{y}{1-y},\frac{x-y}{1-y}) & \mbox{if }x-2y\leq0\mbox{ and }y\leq\frac{1}{2}\\
(\frac{x-y}{1-y},\frac{y}{1-y}) & \mbox{if }x-2y\geq0\mbox{ and }y\leq\frac{1}{2}
\end{cases}.
\]

We will show that the fully subtractive map is equivalent
to the Combo TRIP  Map generated by the permutations $((1\mbox{ }2\mbox{ }3),e,e)$
and the TRIP Map generated by the permutations $((2\mbox{ }3),(1\mbox{ }3\mbox{ }2),(2\mbox{ }3)).$
Described in terms of TRIP Maps, the fully subtractive algorithm first
applies $T_{((1\mbox{ }2\mbox{ }3),e,e)}$ and then applies $T_{((2\mbox{ }3),(1\mbox{ }3\mbox{ }2),(2\mbox{ }3))}$
if $(x,y)$ is not in $\triangle_{1}.$

Now, the set $\{(x,y)|x-2y\leq0\mbox{ and }y\leq\frac{1}{2}\}\cap\triangle$
is the triangle with vertices $(0,0),$ $(1,\frac{1}{2}),$ and $(\frac{1}{2},\frac{1}{2}),$
which will be denoted $\triangle_{2}.$ The set $\{(x,y)|x-2y\geq0\mbox{ and }y\leq\frac{1}{2}\}\cap\triangle$
is the triangle with vertices $(0,0),$ $(1,\frac{1}{2}),$ and $(1,0),$
which will be denoted $\triangle_{1}.$ The set $\{(x,y)|y\geq\frac{1}{2}\}\cap\triangle$
is the triangle with vertices $(\frac{1}{2},\frac{1}{2}),$ $(1,\frac{1}{2}),$
and $(1,1),$ which will be denoted $\triangle_{0}.$ Consider the TRIP
Map generated by the permutations $((1\mbox{ }2\mbox{ }3),e,e)$ and
the TRIP Map generated by the permutations $((2\mbox{ }3),(1\mbox{ }3\mbox{ }2),(2\mbox{ }3)).$
Note that
\begin{equation}
\left[B\left(\begin{array}{ccc}
0 & 1 & 0\\
0 & 0 & 1\\
1 & 0 & 0
\end{array}\right)A_{0}\left(\begin{array}{ccc}
1 & 0 & 0\\
0 & 1 & 0\\
0 & 0 & 1
\end{array}\right)\right]^{T}=\left(\begin{array}{ccc}
1 & 0 & 0\\
1 & 1 & 0\\
2 & 2 & 1
\end{array}\right)
\end{equation}

and
\begin{equation}
\left[B\left(\begin{array}{ccc}
0 & 1 & 0\\
0 & 0 & 1\\
1 & 0 & 0
\end{array}\right)A_{1}\left(\begin{array}{ccc}
1 & 0 & 0\\
0 & 1 & 0\\
0 & 0 & 1
\end{array}\right)\right]^{T}=\left(\begin{array}{ccc}
1 & 1 & 1\\
1 & 0 & 0\\
2 & 2 & 1
\end{array}\right).
\end{equation}

Applying $T_{((2\mbox{ }3),(1\mbox{ }3\mbox{ }2),(2\mbox{ }3))}$
to the second matrix yields
\begin{equation}
\left[\left(\begin{array}{ccc}
1 & 1 & 1\\
1 & 0 & 0\\
2 & 2 & 1
\end{array}\right)^{T}\left(\begin{array}{ccc}
1 & 0 & 0\\
0 & 0 & 1\\
0 & 1 & 0
\end{array}\right)A_{0}\left(\begin{array}{ccc}
0 & 0 & 1\\
1 & 0 & 0\\
0 & 1 & 0
\end{array}\right)\right]^{T}=\left(\begin{array}{ccc}
1 & 0 & 0\\
2 & 1 & 1\\
2 & 2 & 1
\end{array}\right)
\end{equation}

and
\begin{equation}
\left[\left(\begin{array}{ccc}
1 & 1 & 1\\
1 & 0 & 0\\
2 & 2 & 1
\end{array}\right)^{T}\left(\begin{array}{ccc}
1 & 0 & 0\\
0 & 0 & 1\\
0 & 1 & 0
\end{array}\right)A_{1}\left(\begin{array}{ccc}
1 & 0 & 0\\
0 & 0 & 1\\
0 & 1 & 0
\end{array}\right)\right]^{T}=\left(\begin{array}{ccc}
1 & 1 & 1\\
2 & 1 & 1\\
2 & 2 & 1
\end{array}\right),
\end{equation}

so (5), (7), and (8) describe the vertices of $\triangle_{1},$ $\triangle_{0},$
and $\triangle_{2},$ respectively, which is the same division as the
fully subtractive map.

On $\triangle_{0},$ we thus have
\[
\left(\left[B\left(\begin{array}{ccc}
0 & 1 & 0\\
0 & 0 & 1\\
1 & 0 & 0
\end{array}\right)A_{1}\left(\begin{array}{ccc}
1 & 0 & 0\\
0 & 0 & 1\\
0 & 1 & 0
\end{array}\right)A_{0}\left(\begin{array}{ccc}
0 & 0 & 1\\
1 & 0 & 0\\
0 & 1 & 0
\end{array}\right)\right]^{T}\right)^{-1}B^T=\left(\begin{array}{ccc}
1 & 0 & 0\\
0 & 0 & 1\\
-1 & 1 & -1
\end{array}\right),
\]

so
\[
(x,y)\mapsto\left(\frac{y}{1-y},\frac{x-y}{1-y}\right),
\]

which agrees with the definition of the fully subtractive map above.

On $\triangle_{1},$ we thus have
\[
\left(\left[B\left(\begin{array}{ccc}
0 & 1 & 0\\
0 & 0 & 1\\
1 & 0 & 0
\end{array}\right)A_{0}\right]^{T}\right)^{-1}B^T=\left(\begin{array}{ccc}
1 & 0 & 0\\
0 & 1 & 0\\
-1 & -1 & 1
\end{array}\right),
\]

so
\[
(x,y)\mapsto\left(\frac{x-y}{1-y},\frac{y}{1-y}\right),
\]

which agrees with the definition of the fully subtractive map above.

On $\triangle_{2},$ we thus have
\[
\left(\left[B\left(\begin{array}{ccc}
0 & 1 & 0\\
0 & 0 & 1\\
1 & 0 & 0
\end{array}\right)A_{1}\left(\begin{array}{ccc}
1 & 0 & 0\\
0 & 0 & 1\\
0 & 1 & 0
\end{array}\right)A_{1}\left(\begin{array}{ccc}
1 & 0 & 0\\
0 & 0 & 1\\
0 & 1 & 0
\end{array}\right)\right]^{T}\right)^{-1}B^T=\left(\begin{array}{ccc}
0 & 1 & 0\\
0 & 0 & 1\\
1 & -1 & -1
\end{array}\right),
\]

so
\[
(x,y)\mapsto\left(\frac{1-y}{y},\frac{x-y}{y}\right),
\]

which agrees with the definition of the fully subtractive map above.  Hence the fully subtractive map is indeed our desired Combo TRIP map.

  \end{ex}

  \begin{ex}{The G\"uting Map with link to Jacobi-Perron }

The G\"{u}ting map is defined for a point  $(x,y)\in \triangle$ as follows.  First choose
$a$ to be the positive integer so that 
$$1-ax\geq 0 > 1-(a+1)x.$$
Then choose $b$ to be the nonnegative integer so that 
$$1-ax-by\geq 0 > 1-ax-(b+1)y.$$
The G\"{u}ting map is 
$$T(x,y)= \left(\frac{y}{x},\frac{1-ax-by}{x}  \right) .$$

This map does not map each subtriangle bijectively into
$\triangle,$ so we will not concern ourselves with the vertices of triangles
defined by the G\"{u}ting map matching up with those defined by TRIP
maps. Since $T$ maps $\triangle$ into $\triangle,$ all we really have
to do is show that there exists a combination of TRIP maps that agrees
with $T.$ Consider the TRIP Map generated by the permutations $(e,e,e)$
and the TRIP  map generated by the permutations $(e,(1\mbox{ }2\mbox{ }3),e).$

The G\"{u}ting map  can be recast as a combo TRIP map, since by calculation we have 
$$\left(\begin{array}{ccc}
0 & 0 & 1\\
1 & 0 & -a\\
0 & 1 & -b
\end{array}\right) = \left[     B A_0^{-1}  A_1^{-b}      \left(  A_1 \left(\begin{array}{ccc}
1 & 0 & 0\\
0& 0 & 1\\
0 & 1 & 0
\end{array}\right)   \right)^{-2(a-1)} B^{-1}                          \right]^T.$$

The Jacobi-Perron algorithm is connected to the G\"{u}ting map, and thus to TRIP Maps, in a way that the following theorem, found in \cite{SchweigerF00}, makes precise.

\begin{thm}
Let
\begin{equation*}
	G(x_1,x_2)=\left(\frac{x_2}{x_1},\frac{1-a_1x_1-a_2x_2}{x_1}\right)
\end{equation*}
and
\begin{equation*}
	T(x_1,x_2)=\left(\frac{x_2}{x_1}-k_1,\frac{1}{x_1}-k_2\right)
\end{equation*}
denote the maps for the G\"{u}ting algorithm and the Jacobi-Perron algorithm respectively. Then
$$(k^{(g)}_1,k^{(g)}_2)=(a^{(g-n+1)}_2,a^{(g)}_1),$$
when $a^{(t)}_j=0$ for $t\leq 0$ and $j=1,2.$
\end{thm}

This theorem implies that the integer sequence for $(x,y)\in\triangle$ is eventually periodic under the Jacobi-Perron Algorithm if and only if it is periodic under the G\"uting Algorithm.

  \end{ex}

  \section{Periodicity and Cubicness} \label{periodicity}
  
  \subsection{Uniqueness and Cubicness}
  
  We now return to the Hermite problem.  In \cite{SMALL-Uniqueness}, we show 
  \begin{thm}Let $K$ be a cubic number field, $u$ be a real unit  in $\mathcal{O}_K,$ with $0<u<1$ and $E=K(\sqrt{\Delta_{\mathbb{Q}(u)}})$ (where  $\Delta_{\mathbb{Q}(u)}$ is the discriminant of $\mathbb{Q}(u)$).  Then there exists a point $(\alpha,\beta)$, with $\alpha,\beta\in E$ irrational, such that $(\alpha,\beta)$ has a periodic TRIP sequence. 
  \end{thm}
  Actually, in  \cite{SMALL-Uniqueness} the above is a corollary to an explicit listing of five TRIP maps from which, in particular ways, the desired TRIP maps are formed.

  Here, we simply show how (eventual) periodicity for a TRIP map or for a Combo TRIP map for a point  $(\alpha, \beta)$ implies that $\alpha$ and $\beta$ are algebraic in the same number field of degree no more than three, subject to one additional assumption about uniqueness for the TRIP map (which will be the main topic of the next section).
  
  First, given a triangle tree sequence $(i_0,i_1,\ldots )$ with respect to some $(1, \alpha,\beta)$, define $$\triangle(i_0,i_1,\ldots ) = \bigcap_{n\geq 0} \triangle(i_0,i_1,\ldots ,i_n).$$
  
  We use the usual notation $\overline{i_0,i_1,\ldots ,i_l}$ to refer to the repetition of the finite sequence $i_0,i_1,\ldots ,i_l$
  
  We have
  \begin{thm} Suppose $(j_1, \ldots, j_n, \overline{i_0,i_1,\ldots ,i_l})$ is an eventually  periodic sequence such that $$\triangle(j_1, \ldots, j_n,\overline{i_0,i_1,\ldots ,i_l})=\{(1, \alpha, \beta)\}$$ is a single point. Then $\alpha$ and $\beta$ are contained in the same number field of degree at most $3$.

  \end{thm}
  
  This follows from 
  
 \begin{Proposition}
Suppose $(\overline{i_0,i_1,\ldots ,i_l})$ is a periodic sequence such that $$\triangle(\overline{i_0,i_1,\ldots ,i_l})=\{(1, \alpha, \beta)\}$$ is a single point. Then $\alpha$ and $\beta$ are contained in the same number field of degree at most $3$.
\begin{proof}
Since $(\overline{i_0,i_1,\ldots ,i_l})$ is purely periodic, $(1, \alpha, \beta) \cdot F_0F_1\cdots F_{i_l}$ has sequence $(\overline{i_0,i_1,\ldots ,i_l})$. Since the set of points with sequence $(\overline{i_0,i_1,\ldots ,i_l})$ is one-dimensional, $(1, \alpha, \beta)$ is an eigenvector of $F_0F_1\cdots F_{i_l}$.  Further, the matrix  $F_0F_1\cdots F_{i_l}$ has a single largest eigenvalue, whose eigenvector is $(1, \alpha, \beta).$  Since  $F_0F_1\cdots F_{i_l} \in SL(3, \mathbb{Z})$, then by a calcuation we have that  $\alpha$ and $\beta$ are contained in a quadratic or cubic number field.
\end{proof}
\end{Proposition}

Unfortunately, not every eventually periodic sequence for every TRIP map determines a unique point.   In fact, much of the point of \cite{GarrityT05} was showing that we do have $\triangle(\overline{i_0,i_1,\ldots ,i_l}) $  being a unique point for the original triangle map, which recall is TRIP map $T_{(e,e,e)}.$      The goal of the next section is to determine for which TRIP sequences are we guaranteed that periodicity  implies that $\triangle(\overline{i_0,i_1,\ldots ,i_l}) $   is unique.  

\subsection{Some Examples of Periodicity}

For many TRIP maps, though, for specific examples of periodicity we can explicitly link the periodic sequence with algebraic properties of $\alpha$ and $\beta$.

In this subsection, we'll look at some periodicity properties for some specific TRIP Maps. We will do several examples illustrating periodic points and
also show how the addition of other maps improves previous results
regarding the triangle map.
\begin{ex}
(This was in  \cite{GarrityT01})
Let $\alpha$ be the root of $x^{3}+x-1$ that lies between 0 and
1. By the Intermediate Value Theorem, $\alpha$ exists. Consider the
triangle sequence of $\overline{\alpha}=(\alpha,\alpha^{2}).$ Since $(\alpha,\alpha^{2})$
lies in $\triangle_{0},$ the first integer in $\overline{\alpha}$'s
triangle sequence is 0. Now, we apply the triangle map to
$\overline{\alpha}.$ 
\[
T_{(e,e,e)}(\alpha,\alpha^{2})=\Big(\frac{\alpha^{2}}{\alpha},\frac{1-\alpha}{\alpha}\Big)
\]
\[
=\Big(\alpha,\frac{\alpha^{3}}{\alpha}\Big)
\]
\[
=(\alpha,\alpha^{2}).
\]

Thus, since $\overline{\alpha}$ is a fixed point of $T_{(e,e,e)},$
every iteration of the triangle map lies in $\triangle_{0},$ and the
triangle sequence of $\overline{\alpha}$ is $(0,0,0,\ldots)$.
\end{ex}
This example emphasizes  that to understand periodicity with respect to TRIP Maps, it is
simpler to look at eigenvectors of matrices. In fact,
the example above can be done using the fact that $(1,\alpha,\alpha^{2}),$
as defined above, is an eigenvector of the matrix representation of
$T_{(e,e,e)}^{0}.$ In general, if $(1,\alpha,\beta)$ is an eigenvector
of the matrix representing some triangle map $T_{(\sigma,\tau_{0},\tau_{1})}$ and $(\alpha,\beta)\in\triangle,$
then $T_{(\sigma,\tau_{0},\tau_{1})}$ maps $(1,\alpha,\beta)$
to itself. Thus, $(1,\alpha,\beta)$ has a purely periodic trip
sequence of period 1. This result holds for purely periodic TRIP
sequences of any length. Next, we'll do another example of pure periodicity
using a different triangle map from two different approaches: the algebraic manipulation approach and the eigenvector approach.
\begin{ex}
Let $\alpha$ be the root of $2x^{3}-5x^{2}+x+1=0$ that lies in $(0,1).$
By the Intermediate Value Theorem, $\alpha$ exists. Consider the
trip sequence of $\overline{\alpha}=(\alpha,2\alpha-2\alpha^2)$ under the map
$T_{((2\ 3),e,(2\ 3))}.$ We know that $\overline{\alpha}$ lies in
$\triangle_{1},$ so the first integer of $\overline{\alpha}$'s triangle
sequence is 1. Now, we can apply $T_{((2\ 3),e,(2\ 3))}$ to $\overline{\alpha}.$
\[
T_{((2\ 3),e,(2\ 3))}(\alpha,2\alpha-2\alpha^2)=(2\alpha-1 ,\frac{1}{\alpha}-2\alpha).
\]

We know that $(2\alpha-1 ,\frac{1}{\alpha}-2\alpha)$ lies in $\triangle_{2},$ so the second
integer of $\overline{\alpha}$'s TRIP sequence is 2. Now, we
can apply $T_{((2\ 3),e,(2\ 3))}$ to $(2\alpha-1 ,\frac{1}{\alpha}-2\alpha).$
\[
T_{((2\ 3),e,(2\ 3))}(2\alpha-1 ,\frac{1}{\alpha}-2\alpha)=\left(2+\frac{1}{\alpha}+\frac{1}{1-2\alpha},\frac{3}{2\alpha-1 }-5-\frac{2}{\alpha}\right),
\]

which, after some messy simplification, equals $(\alpha,2\alpha-2\alpha^2).$ Hence, the TRIP sequence of $\overline{\alpha}$ is $1,2,1,2,1,2\dots.$
\end{ex}
Now, we'll use eigenvectors to solve the same problem.
\begin{ex}
Consider the map $T_{((2\ 3),e,(2\ 3))}\circ T_{((2\,3),e,(2\,3))}$ with $k=1$ for the map on the left and $k=2$ for the map on the right.
This corresponds to the matrix
\[
M=\left(\begin{array}{ccc}
0 & -1 & 2\\
1 & 3 & -6\\
-1 & -2 & 5
\end{array}\right).
\]

$M$ has $(1,\alpha,2\alpha-2\alpha^2),$ as defined in the previous
example, as an eigenvector. Thus, the Trip sequence of $\overline{\alpha}$
under the map $T_{((2\ 3),e,(2\ 3))}$ is $1,2,1,2,1,2\ldots .$
\end{ex}
As we can see in the previous two examples, thinking of periodicity
in terms of eigenvectors greatly simplifies computation. In addition,
it can be used to prove a variety of theorems concerning periodicity.
In  \cite{GarrityT01} it was shown for the triangle
map that 
\begin{thm}
We have that $(\alpha,\alpha^{2})$ has a purely periodic triangle
sequence of period one if and only if $\alpha^{3}+k\alpha^{2}+\alpha-1=0$
for some $k\in\mathbb{N}.$
\end{thm}
We can combine two different TRIP Maps, $T_{(e,e,e)}$ and $T_{((1\ 3\ 2),(1\ 3\ 2),e)}$,
to extend this theorem.
\begin{thm}
Let $A,B\in\mathbb{Z}$ with $A\geq0$ and $B\geq1.$ Then, if $\alpha^{3}+A\alpha^{2}+B\alpha-1=0,$
$(\alpha,\alpha^{2})$ has a periodic triangle sequence under a combination
of the maps $T_{(e,e,e)}\mbox{ and }T_{((1\,3\,2),(1\,3\,2),e)}.$\end{thm}
\begin{proof}
We know that $ T_{((1\,3\,2),(1\,3\,2),e)},$ when $k=1$, corresponds to the matrix
\[
M_{1}=\left(\begin{array}{ccc}
1 & 0 & 0\\
-1 & 1 & 0\\
0 & 0 & 1
\end{array}\right),
\]

and
\[
M_{1}^{B}=\left(\begin{array}{ccc}
1 & 0 & 0\\
-B & 1 & 0\\
0 & 0 & 1
\end{array}\right).
\]

We also know that $T_{(e,e,e)}$ corresponds to the matrix when $k=A$
\[
M_{2}=\left(\begin{array}{ccc}
0 & 0 & 1\\
1 & 0 & -1\\
0 & 1 & -A
\end{array}\right).
\]

The product $M_{1}^{B}M_{2}$ is 
\[
\left(\begin{array}{ccc}
0 & 0 & 1\\
1 & 0 & -B\\
0 & 1 & -A
\end{array}\right),
\]

which has eigenvector $(1,\alpha,\alpha^{2})$ where $\alpha^{3}+A\alpha^{2}+B\alpha-1=0,$
as desired.\end{proof}

While the results in this section show that certain pairs in $\triangle$ have a purely periodic TRIP sequence by virtue of being of some form, it is not always the case that the converse is true. Here, we run into the issue of uniqueness of trip sequences, which we now treat in  depth.

\section{Uniqueness}\label{uniqueness}
We have just seen that the eventual periodicity of a  TRIP or Combo TRIP sequence implies that $\alpha$ and $\beta$ are algebraic in the same number field of degree less than or equal to three, provided the particular triangle tree sequence defines a unique point.  Unfortunately,  TRIP sequences need not always define unique points, as shown in \cite{GarrityT05} for the triangle map, $T_{(e,e,e)}$.  In this section we give a condition that determines when a TRIP sequence  will define a single point.  This will allow us to apply this criterion to periodic triangle tree sequences. 

\subsection{General Uniqueness Results}\label{generaluniqueness}
Given a TRIP tree sequence $(i_0,i_1,\ldots )$ with respect to some $(\sigma, \tau_0,\tau_1)$, recall that  $\triangle(i_0,i_1,\ldots ) = \bigcap_{n\geq 0} \triangle(i_0,i_1,\ldots ,i_n).$ The following lemma states that the cardinality of $\triangle(i_0,i_1,\ldots )$ does not depend on the first finitely many terms of $(i_0,i_1,\ldots )$.
\begin{lemma} \label{init}
The intersection $\triangle(i_0,i_1,\ldots )$ is a unique point if and only if the intersection $\triangle(i_k,i_{k+1},\ldots )$ is a unique point for any $k \in \mathbb{N}$.
\begin{proof}
The product $F_{i_0}F_{i_1}\cdots F_{i_k}$ gives a bijection between $\triangle$ and $\triangle(i_0,i_1,\ldots ,i_k)$ that takes $\triangle(i_0,i_1,\ldots )$ to $\triangle(i_k, i_{k+1},\ldots )$.
\end{proof}
\end{lemma}

Let $X_{n, 1} = (x_{n,1},y_{n,1},z_{n,1}),$ $X_{n,2} = (x_{n,2},y_{n,2},z_{n,2})$, and $X_{n, 3} =(x_{n,3},y_{n,3},z_{n,3})$ be the three vertices of $\triangle(i_0,i_1,\ldots ,i_n)$, given in Cartesian coordinates. The vertices of the triangle after projection are $\hat{X}_{n,i} = (\frac{y_{n,i}}{x_{n,i}},\frac{z_{n,i}}{x_{n,i}}).$

The goal of this subsection is 
\begin{thm} \label{ratios}
Let $(i_0,i_1,\ldots )$ be a TRIP tree sequence with respect to any combination of maps. Suppose there exists a constant $C$ such that for every $n \geq 0$ and every pair $1 \leq i, j \leq 3$,
$$\frac{x_{n,i}}{x_{n,j}} \leq C.$$
Then the intersection $\triangle(i_0,i_1,\ldots )$ is a unique point.
\end{thm}

Results similar to this theorem are given in \cite{Kerckhoff85} in the context of interval exchange maps. The proof given here is new, as are the applications to multidimensional continued fractions.
 
The converse is not true. There exist sequences $(i_0,i_1,\ldots )$ such that $\triangle(i_0,i_1,\ldots )$ is a single point, but $\frac{x_{n,i}}{x_{n,j}}$ is not bounded. See \cite{GarrityT05} for examples, including the triangle sequences $a_n = n$ and $a_n = n^2$.

Before giving the proof, we will need two lemmas. The first is true for any triangle.
\begin{lemma} \label{dist}
For any triangle $\triangle(i_0,i_1,\ldots ,i_n)$,
the sum $\widehat{X_{n,i} + X_{n,j}}$ is a weighted average of $\hat{X}_{n,i}$ and $\hat{X}_{n,j}$ in the sense that
$$d(\hat{X}_{n, i}, \widehat{X_{n, i} + X_{n, j}}) = \frac{x_{n,j}}{x_{n,i} + x_{n,j}} d(\hat{X}_{n, i}, \hat{X}_{n,j}).$$
\begin{proof}
\begin{align*}
d(\hat{X}_{n, i}, \widehat{X_{n, i} + X_{n, j}}) & = ||\frac{X_{n,i}}{x_{n,i}} - \frac{X_{n,i}+X_{n,j}}{x_{n,i} + x_{n, j}}||
\\ & = || \frac{x_{n,j}}{x_{n,i}(x_{n, i} + x_{n,j})} X_{n, i} -\frac{1}{x_{n, i} + x_{n,j}} X_{n, j}||
\\ & =  \frac{x_{n, j}}{x_{n,i} + x_{n,j}}d(\hat{X}_{n,i},\hat{X}_{n,j}).
\end{align*}
\end{proof}
\end{lemma}

\begin{lemma}\label{change} Let $(i_0,i_1,\ldots )$ be a triangle tree sequence with respect to any combination of maps. Suppose there exists a constant $C$ such that for every $n \geq 0$ and every pair $1 \leq i, j \leq 3$,
$\frac{x_{n,i}}{x_{n,j}} \leq C.$  Then if $X_{n,i}$ is a vertex of $\triangle(i_0,i_1,\ldots ,i_n)$, then it cannot be a vertex of $\triangle(i_0,i_1,\ldots ,i_{n+k})$ for any $k>2C^2.$

  \begin{proof}
  Let $X_{n, 1}$ be a vertex of $\triangle(i_0,i_1,\ldots , i_n)$.  Suppose that $X_{n, 1}$ remains a vertex for any $\triangle(i_0,i_1,\ldots ,i_{n+k})$.  We have that 
  $$\frac{x_{n,1}}{\min(x_{n,1},x_{n,2}, x_{n,3})}\leq C.$$
    Thus the minumum of the $x_{n,1},$ $x_{n, 2},$ and $x_{n, 3}$ is at least $\frac{x_{n,1}}{C}.$
     When we form new subdivisions, there will be a new vertex one of whose terms is the sum of positive multiples of $x_{n,2}$ with positive multiples of $x_{n,3}$.  Thus, if $X_{n, 1}$ remains a vertex for  $\triangle(i_0,i_1,\ldots ,i_{n+k})$, one of the other vertices must have a term that is at least $kx_{n,1}/C$.  Suppose that $k>2C^2.$  Then this new vertex must have a term that is at least $2Cx_{n,1}.$
  Then we have 
  $$2C = \frac{2Cx_{n,1}}{x_{n,1} }\leq \frac{\mbox{term of new vertex}}{x_{n,1}} \leq C,$$
  giving us our contradiction.

  \end{proof}
\end{lemma}

We are now ready to prove the theorem.

\begin{proof}  To prove uniqueness, we will show that the lengths of the sides of  $\triangle(i_0,i_1,\ldots ,i_{n})$ go to zero as $n$ goes to infinity.
 From Lemma  \ref{dist} we know that 
$$d(\hat{X}_{n, i}, \widehat{X_{n, i} + X_{n, j}}) = \frac{x_{n,j}}{x_{n,i} + x_{n,j}} d(\hat{X}_{n, i}, \hat{X}_{n,j}).$$
By assumption, we have 
$$\frac{x_{n,j}}{C} \leq x_{n,i},$$
which means that 
$$ \frac{x_{n,j}}{x_{n,i} + x_{n,j}}  \leq  \frac{x_{n,j}}{\frac{x_{n,j}}{C} +x_{n,j}} =\frac{C}{C+1},$$
which in turns means that 
$$d(\hat{X}_{n, i}, \widehat{X_{n, i} + X_{n, j}}) = \frac{C}{C+1} d(\hat{X}_{n, i}, \hat{X}_{n,j}).$$

   Let  $\hat{X}_{n,1}$, $\hat{X}_{n,2} $ and $\hat{X}_{n,3}$  be the vertices for $\triangle(i_0,i_1,\ldots ,i_n)$.  Without loss of generality, we can assume that the vertices for  $\triangle(i_0,i_1,\ldots ,i_{n+1})$ are  $\hat{X}_{n,1}$, $\hat{X}_{n,2} $ and $\widehat{X_{n,2} + X_{n,3}}$. Then we know
   $$d(\hat{X}_{n, 2}, \widehat{X_{n, 2} + X_{n, 2}}) \leq \frac{C}{C+1} d(\hat{X}_{n, 2}, \hat{X}_{n,3}).$$
   This certainly gives us that, as $n\rightarrow \infty$, the length of one of the sides must approach zero.  
   
   Of course, just because one of the side lengths goes to zero does not mean that all side lengths approach zero.  This will happen, though, if we can show the lengths of  two of the sides go to zero.  Here Lemma \ref{change} becomes important.  The only way for just one side to approach zero is for the vertex opposite that side to not change.  In the above paragraph, this would mean that the vertex $\hat{X}_{n,1}$ of $\triangle(i_0,i_1,\ldots ,i_n)$ must remain a vertex for all subsequent $\triangle(i_0,i_1,\ldots ,i_{n+k}),$ which we have seen is impossible.  Hence we have our result.

\end{proof}

The following Corollary states that for purely periodic sequences, it is sufficient to satisfy the condition once each period.
\begin{cor}
Let $(\overline{i_0,i_1,...,i_l})$ be a purely periodic TRIP tree sequence with respect to any combination of maps. Suppose there exists a constant $C$ and an integer $1 \leq k \leq l$ such that for every $n \geq 0$ and every pair $1 \leq i, j \leq 3,$ $$\frac{x_{nl+k,i}}{x_{nl+k,j}} \leq C.$$
Then the intersection $\triangle(\overline{i_0,i_1,...,i_l})$ is a unique point.
\end{cor}
\begin{proof}
Note that if $$\frac{x_{nl+k,i}}{x_{nl+k,j}} \leq C,$$
then $$\frac{x_{nl+k+1,i}}{x_{nl+k+1,j}} \leq 2C,$$
because each $x_{nl+k+1,i}$ is at most the sum of two $x_{nl+k,j}.$ So for each $m$, $$\frac{x_{m,i}}{x_{m,j}} \leq  2^lC.$$
Apply Theorem 7.2.
\end{proof}

\subsection{Uniqueness for Periodic Sequences}\label{uniquenessperiodic}

In the last section we showed that when a  periodic TRIP sequences determines a unique point, the coordinates of that point are rational, quadratic, or cubic.
Unfortunately, not every periodic TRIP sequence determines a unique point.
\begin{ex}
Let $(\sigma, \tau, \tau_1) = (e, (1 2), e)$. Then $F_0$ sends $(v_1,v_2,v_3)$ to $(v_3,v_2,v_1 + v_3)$ and $F_1$ sends $(v_1,v_2,v_3)$ to $(v_1, v_2, v_1 + v_3)$. The vertices of $\triangle(i_0,i_1,\ldots ,i_n)$ are of the form $(a v_1 + bv_3, v_2, c v_1 + dv_3)$, where $a,b,c,d \in \mathbb{Z}$. So the second vertex of $\triangle(i_0,i_1,\ldots ,i_n)$ is the bottom left corner of $\triangle$. The first and third vertices are points on the hypotenuse of $\triangle$ with rational coordinates.

Suppose $\triangle(i_0,i_1,\ldots ,i_n)$ contains any point $(x, y)$. Then $\triangle(i_0,i_1,\ldots ,i_n)$ must contain all points on the line containing $(1, 0)$ and $(x, y)$ that are inside the original triangle. The line segment is given by $$L = \{(ax, (1-a) + ay): a \in \mathbb{R}, 0 \leq ax \leq (1 - a) + ay \}.$$ So if $(x, y)$ is in $\displaystyle \lim_{n \rightarrow \infty} \triangle(i_0,i_1,\ldots ,i_n)$, then $L$ is as well. In particular, every periodic sequence specifies an entire line. Each line must contain points $(\alpha, \beta)$ such that $\alpha$ and $\beta$ are transcendental.

Further, this line need not be the boundary of any of our triangles $\triangle$. The triangle sequence $(1,1,1,\ldots )$ specifies the line segment from $(1,0)$ to $(\alpha, \alpha^2)$, where $\alpha$ is the root of $x^3 -2x^2 -x+1$ in the unit interval. This segment has irrational slope.
\end{ex}
We want to determine when a periodic TRIP sequence determines a unique point. By lemma \ref{init}, it is sufficient to consider purely periodic TRIP sequences.

\begin{lemma}
Suppose theTRIP tree sequence $\triangle(i_0,i_1,\ldots )$ determines a unique point with respect to the permutations $(\sigma, \tau_0, \tau_1).$ Then $\triangle(i_0,i_1,\ldots )$ determines a unique point with respect to $(\rho \sigma, \tau_0 \rho^{-1}, \tau_1 \rho^{-1})$ for every $\rho \in S_3$.
\begin{proof}
The two subdivisions of $\triangle$ are equivalent, up to a permutation of $v_1$, $v_2,$ and $v_3$. The cardinality of $\triangle(i_0,i_1,\ldots )$ does not depend on the choice of initial vertices.
\end{proof}
\end{lemma}
We need only classify our maps in the case that $\sigma = e$, where $e \in S_3$ is the identity permutation.
\begin{thm} \label{evs}
Let $(\overline{i_0,i_1,\ldots ,i_l})$ be a purely periodic TRIP tree sequence. The intersection $\triangle(i_0,i_1,\ldots )$ contains a line segment of non-zero length if and only if $(F_{i_0}F_{i_1}\ldots F_{i_l})^2$ has two eigenvectors contained in the triangle $\triangle$.
\begin{proof}
Suppose $(F_{i_0}F_{i_1}\cdots F_{i_l})^2$ has two eigenvectors, $w_1$ and $w_2$, contained in $\triangle.$ Suppose $w$ is on the line segment between $w_1$ and $w_2$. There exists $0 \leq \mu \leq 1$ such that \mbox{$w = \mu w_1 + (1 - \mu) w_2$}. If $\lambda_1$ and $\lambda_2$ are the eigenvalues of $w_1$ and $w_2$, then
$$w(F_{i_0}F_{i_1}\cdots F_{i_l})^{2k} = \lambda_1^{2k} \mu w_1 + \lambda_2^{2k}(1 - \mu) w_2$$
is on the line segment from $w_1$ to $w_2$. 
So $w$ has triangle tree sequence $(\overline{i_0,i_1,..,i_l})$.

Suppose $\triangle(\overline{i_0,i_1,\ldots ,i_l})$ is the line segment from $w_1$ to $w_2$. Because $\overline{i_0,i_1,\ldots ,i_l}$ is purely periodic,
$$\big(\triangle(\overline{i_0,i_1,\ldots ,i_l})\big)F_{i_0}F_{i_1}\cdots F_{i_l} = \triangle(\overline{i_0,i_1,\ldots ,i_l}).$$
Because $F_{i_0}F_{i_1}\cdots F_{i_l}$ is linear, the map takes the line segment from $w_1$ to $w_2$ to the line segment from $w_1F_{i_0}F_{i_1}\cdots F_{i_l}$ to the line segment from $w_2F_{i_0}F_{i_1}\cdots F_{i_l}.$ The map $F_{i_0}F_{i_1}\cdots F_{i_l}$ must fix or permute $w_1$ and $w_2$. In either case, $w_1$ and $w_2$ are eigenvectors of $(F_{i_0}F_{i_1}\cdots F_{i_l})^2$.
\end{proof}
\end{thm}

\begin{Proposition}
Suppose $(\sigma, \tau_0, \tau_1)$ is contained in the list
\[
\begin{array}{llll}
(e, (1\,2), e),& (e, (1\,2), (1\,2)),& (e, (1\,2), (2\,3)),& (e, (1\,2) , (1\,3\,2)), \\
(e, (1\,2), (1\,2\,3)),& (e, (1\,2), (1\,3)),& (e, (1\,2\,3), (1\, 3)),& (e, e,(1\,3)), \\
(e,(1\,3\,2),(1\,3)),& (e, (2\,3), (1\,3)),& (e, (1\,3), (1\,3)),& (e, e, (2\,3)), \\
(e, e, (1\,2)),& (e, (2\,3), (1\,2\,3)),& (e, (1\,2\,3),e),& (e, (1\,3), (1\,2\, 3)), \\
(e, (1\,3), (2\,3)),& (e, e, (1\,2\,3)),& (e,(1\,3),e),& (e, (1\,3), (1\,2)), \\
(e,(1\,3\,2), (1\,2)),& (e, (1\,2\,3), (1\,2)), & (e,(1\,3),(1\,3\,2)), & (e,(2\,3),(1\,2)), \\
(e,e,(1\,3\,2), & (e, (1\,3\,2),(1\,2\,3)).
\end{array}
\]
There exists a purely periodic TRIP sequence $(\overline{i_0,i_1,\ldots ,i_l})$ such that $\triangle(\overline{i_0,i_1,\ldots ,i_l})$ is a line segment with respect to $(\sigma, \tau_0, \tau_1)$. This line segment is not contained in the boundary of $\triangle$.
\begin{proof}
Using Proposition \ref{evs}, we can exhibit specific examples of such sequences in the above cases. 
For each of the  cases $(e, (1\,2), e)$,  $(e, (1\,2), $ $(1\,2)), $ $  (e, (1\,2),$ $ (2\,3)), $ $  (e, (1\,2) , $ $(1\,3\,2)),$ $ (e, (1\,2),$ $  (1\,2\,3)), $ $ (e, (1\,2), $ $ (1\,3))$. the matrix $F_0$ has the two desired eigenvectors in $\triangle$ and hence for all of theses cases there is an entire line seqment with   sequence $(\overline{0})$.  For each of the cases $(e, (1\,2\,3)$ , $(1\, 3)), (e, e,(1\,3))$, 
$(e,(1\,3\,2),(1\,3)), $ $(e, (2\,3), (1\,3)),  (e, (1\,3), (1\,3))$, the matrix $F_1$ has the two desired eigenvectors in $\triangle$ and hence for all of theses cases there is an entire line seqment with   sequence $(\overline{1})$.    For each of the cases $(e, e, (2\,3)),$ $(e, e, (1\,2)),$ $(e, (2\,3), (1\,2\,3))$ and $(e, (1\,2\,3), e$,  the matrix $F_0F_1$ has the two desired eigenvectors in $\triangle$ and hence for all of theses cases there is an entire line seqment with   sequence $(\overline{0,1})$.  For each of the two  cases $(e, (1\,3), (1\,2\,3))$ and $(e, (1\,3), (2\,3)),$ the matrix $F_0^3F_1$ has the two desired eigenvectors in $\triangle$ and hence for both of theses cases there is an entire line seqment with   sequence $(\overline{0,0,0,1})$.  For each of the three  cases $(e, e, (1\,2\,3))$, $(e, (1\,3), e)$ and $(e, (1\,3), (1\,2))$,  the matrix $F_0F_1^2F_0$ has the two desired eigenvectors in $\triangle$ and hence for these three cases there is an entire line seqment with   sequence $(\overline{0,1,1,0})$.  For the  case $(e, (1\,3\,2), (1\,2))$, the matrix $F_1F_0^2$ has the two desired eigenvectors in $\triangle$ and hence there is an entire line seqment with   sequence $(\overline{1,0,0})$. 
For each of the two  cases $(e, (1\,2\,3), (1\,2))$ and $(e, (1\,3), (1\,3\,2)),$ the matrix $F_1^2F_0$ has the two desired eigenvectors in $\triangle$ and hence for both of theses cases there is an entire line seqment with   sequence $(\overline{1,1,0})$.
For the  case $(e, (2\,3), (1\,2))$, the matrix $F_1^3F_0$ has the two desired eigenvectors in $\triangle$ and hence there is an entire line seqment with   sequence $(\overline{1,1,1,,0})$. 
For the  case $(e, e, (1\,3\,2))$, the matrix $F_1^2F_0^2F_1F_0$ has the two desired eigenvectors in $\triangle$ and hence there is an entire line seqment with   sequence $(\overline{1,1,0,0,1,0})$. 
Finally, for the  case $(e, (1\,3\,2), (1\,2\,3))$, the matrix $F_0^2F_1^2F_0F_1$ has the two desired eigenvectors in $\triangle$ and hence there is an entire line seqment with   sequence $(\overline{0,0,1,1,0,1})$. 

For all of these, the proofs are just calculations.

 \end{proof}
\end{Proposition}
For all of these cases there should be  many other periodic sequences  for which we do not have uniqueness. It would be interesting to be able to predict from the sequence whether or not we get periodicity.

\begin{ex}
Let $(\sigma, \tau_0, \tau_1) = (e, e, (1\ 2)).$ Then $F_0$ and $F_1$ send $(\vv_1,\vv_2,\vv_3)$ to $(\vv_2,\vv_3,\vv_1 + \vv_3)$ and $(\vv_2,\vv_1,\vv_1 + \vv_3)$ respectively. The matrix $F_1F_0$ has eigenvectors
\[\vv_1 \quad\text{and}\quad \frac{1}{2}(3 + \sqrt{5})\vv_1 + \frac{1}{2}(-1 + \sqrt{5})\vv_2 + \vv_3.\]
The intersection $\triangle(\overline{1, 0})$ contains the line segment between these two eigenvectors.
\end{ex}
For the remaining maps, periodic sequences do specify unique points.
\begin{Proposition}
Suppose $(\sigma, \tau_0, \tau_1)$ is contained in the list
\[
\begin{array}{llll}
(e,e,e),& (e, (1\,2\,3), (1\,2\,3)),& (e, (2\,3), e),& (e, (1\,2\,3), (2\,3)), \\
(e, (2\,3), (2\,3)),& (e, (2\,3), (1\,3\,2)),& (e, (1\,3\,2), (2\,3) ),& (e, (1\,3\,2), (1\,3\,2)), \\
(e, (1\,3\,2), e),& (e, (1\,2\,3), (1\,3\,2)).
\end{array}
\]
Choose any purely periodic TRIP sequence $(\overline{i_0,i_1,\ldots ,i_l})$ such that $i_0,i_1,\ldots ,i_l$ are not all $0$ and not all $1$. The intersection $\triangle(\overline{i_0,i_1,\ldots ,i_l})$ is a unique point with respect to $(\sigma, \tau_0, \tau_1)$. This point is of the form $(1, \alpha, \beta)$, where $\alpha$ and $\beta$ are contained in the same quadratic or cubic number field. The intersections $\triangle(\overline{0})$ and $\triangle(\overline{1})$ are unique points or edges of $\triangle$, depending on the choice of $(\sigma, \tau_0, \tau_1)$.
\begin{proof}
 For the following maps, for every purely periodic triangle tree sequence $(\overline{i_0,i_1,...,i_l})$, then $\triangle(\overline{i_0,i_1,...,i_l})$  is a unique point (or an edge of the original $\triangle$). Suppose $(\overline{a_0,....,a_m})$ is the TRIP sequence corresponding to $(\overline{i_0,...,i_l})$. Recall that each $a_i$ counts the number of repetitions of $F_1$ between each application of $F_0$. Let $(x_{k,1},x_{k,2},x_{k,3})$ be the first coordinates of the vertices of the triangle $\triangle(a_0,a_1,...,a_k)$. By Theorem 6.2, it is sufficient to show that the ratios between the $x_i$ are bounded by some constant $C$ independent of $n$.

\begin{itemize}
\item The permutations $(e,e,e)$ give matrices
$$F_0 =  \begin{pmatrix} 0 & 0 &1 \\1 &0&0 \\0& 1&1 \end{pmatrix} \text{ and }F_1 = \begin{pmatrix}1& 0&1  \\ 0&1&0 \\0& 0 & 1 \end{pmatrix}.$$
Recall that $\triangle(a_0,...,a_{k+1})$ is obtained by applying $F_1^{a_k}F_0$ to $\triangle(a_0,...,a_k)$.
Because $$F_1^{a_k}F_0 = \begin{pmatrix} 0 & a_k & a_k+1 \\ 1 & 0 & 0 \\ 0 & 1 & 1\end{pmatrix},$$
we have recursion relations
\begin{align*}(x_{k+1,1},x_{k+1,2},x_{k+1,3}) &= F_1^{a_n}F_0(x_{k,1},x_{k,2},x_{k,3}) \\&= (x_{k,2}, a_k x_{k,1} + x_{k,3}, (a_k + 1)x_{k,1} + x_{k,3}).\end{align*}

First, we will show by induction that $x_{k,1} \leq x_{k,2} \leq x_{k,3}$ for each $k \geq 0$. Because $(x_{1,1},x_{1,2},x_{1,3}) = (1,1,1)$, the base case is clear. Suppose the inequalities hold for $k$. Since the $x_{k,i}$ are all positive, $(a_n + 1)x_{k,1 }+ x_{k,3} > a_n x_{k,1} + x_{k,3}.$ Substitution gives $x_{k+1,3} >x_{k+1,2}$. So the inequality $$x_{k,2} \leq x_{k,3}$$holds for all $k$. Since $x_{k+1,1} = x_{k,2}$ and $x_{k+1,2} \geq x_{k,3}$, it follows that $$x_{k+1,1}\le x_{k+1,2}.$$ The inequalities are established for all $k$.

For each $k$, we claim that the ratios between the $x_{k+1,i}$ are bounded by $2a_k + 4$. The recursion relations give $$\frac{x_{k+1,3}}{x_{k+1,3}} = \frac{(a_k + 1)x_{k,1} + x_{k,3}}{a_kx_{k,1} + x_{k,3}}.$$Because the $x_{k,3}$ terms dominate for small $a_k$, the ratio between $x_{k+1,3}$ and $x_{k+1,2}$ satisfies the bounds: $$1 \leq \frac{(a_k + 1)x_{k,1} + x_{k,3}}{a_kx_{k,1} + x_{k,3}}\le 2.$$
Because $x_{k,2} \leq x_{k,3}$ for $k > 0$, the ratio $$\frac{x_{k+1,2}}{x_{k+1,1}} = \frac{a_kx_{k,1} + x_{k,3}}{x_{k,2}}$$between $x_{k+1,2}$ and $x_{k+1,1}$ satisfies the bounds: $$1 \leq \frac{a_k x_{k,1} + x_{k,3}}{x_{k,2}} \leq a_k + 2.$$
Because the $a_k$ are bounded for a purely periodic sequence, the hypothesis of Theorem 6.2 is satisfied. The remaining case is the sequence is $(\overline{1})$, which defines the lower edge of the triangle.

\item The permutations $(e,(1\,2\,3),(1\,2\,3))$ give the matrices 
$$F_0 = \begin{pmatrix} 1 & 0 & 0 \\ 0 &1 & 0 \\1 & 0 & 1 \end{pmatrix} \text{ and } F_1 = \begin{pmatrix} 1 & 1 & 0 \\ 0 & 0 &1 \\ 1& 0 & 0\end{pmatrix}.$$
If we change from the basis given by $(1,0,0), (0,1,0),$ and $(0,0,1)$ to the basis given by $(0,0,1), (0,1,0),$ and $(1,0,0)$, then $F_0$ and $F_1$ are given by
$$F_0 = \begin{pmatrix}1& 0&1  \\ 0&1&0 \\0& 0 & 1 \end{pmatrix} \text{ and }F_1=  \begin{pmatrix} 0 & 0 &1 \\1 &0&0 \\0& 1&1 \end{pmatrix}.$$
These are the matrices the $F_1$ and $F_0$ for  the permutations $(e,e,e)$, respectively, so the proof proceeds in the same way.

\item The permutations $(e,(2\,3),e)$ give matrices$$F_0 = \begin{pmatrix} 0 & 1 & 0 \\ 1 & 0 & 0 \\ 0 & 1  & 1 \end{pmatrix} \text{ and } F_1 =\begin{pmatrix} 1 & 0 & 1 \\ 0 & 1 & 0 \\ 0 & 0 & 1 \end{pmatrix}$$
Because $$F_1^{a_k}F_0 = \begin{pmatrix} 0 & a_k + 1 & a_k \\ 1 & 0 & 0 \\ 0 & 1 & 1 \end{pmatrix},$$
the recursion relations give $$(x_{k+1,1},x_{k+1,2},x_{k+1,3}) = (x_{k,2}, (a_k+1)x_{k,1}+x_{k,3},a_kx_{k,1}+x_{k,3}).$$
Because $x_{k,1}$ is always positive, $$x_{k+1,2}  = (a_k + 1)x_{k,1} + x_{k,3}> a_k x_{k,1} + x_{k,3} = x_{k+1,3}$$for each $k$. We claim that $x_{k,2} \geq x_{k,1}$ and $x_{k,1}+x_{k,3} \ge x_{k,2}$ as well. The proof is by induction. If $x_{k,2} \geq x_{k,1}$, then $$x_{k+1,1}+x_{k+1,3} = a_kx_{k,1} +x_{k,2}+x_{k,3} \geq(a_k+1)x_{k,1}+x_{k,3} = x_{k+1,2}.$$
Therefore, 
$$x_{k+2,2} = (a_{k+1}+1)x_{k+1,1} + x_{k+1,3}   \geq  x_{k+1,1} + x_{k+1,3}   \geq x_{k+1,2} =x_{k+2,1}.$$
Because the inequalities hold for $k = 0$ and $k = 1$, we have shown $x_{k,2} \geq x_{k,1}$ and $x_{k,1}+x_{k,3} \ge x_{k,2}$ for all $k$.

Suppose $a_k \neq 0$. Combining these inequalities, the ratio
$$\frac{x_{k+2,2}}{x_{k+1,1}} = \frac{(a_k+1)x_{k,1}+x_{k,3}}{x_{k,2}}$$
satisfies the bounds
$$1 \leq \frac{(a_k+1)x_{k,1}+x_{k,3}}{x_{k,2}} \leq a_k + 2.$$
The ratio
$$\frac{x_{k+2,2}}{x_{k+2,3}} = \frac{(a_k+1)x_{k,1}+x_{k,3}}{a_kx_{k,1}+x_{k,3}}$$
is bounded by
$$1 \leq \frac{(a_k+1)x_{k,1}+x_{k,3}}{a_kx_{k,1}+x_{k,3}} \leq a_k + 2$$
as well. So the ratios between the $x_{k,i}$ are bounded by $a_k +2$ whenever $a_k \neq 0$. There is a non-zero $a_k$ each period except for the sequences $(\overline{0})$ and $(\overline{1})$. So the hypothesis of Corollary 7.1 is satisfied. The remaining sequences are $(\overline{0}),$ which is the midpoint of the right edge, and $(\overline{1})$, which is the lower edge.

\item The matrices for $(e,(1\,2\,3),(2\,3))$ are $$F_0 = \begin{pmatrix} 1 & 0 & 0 \\ 0 & 1 & 0 \\ 1 & 0 & 1 \end{pmatrix} \text{ and }F_1= \begin{pmatrix} 1 & 1 & 0 \\ 0 & 0 & 1 \\ 0 & 1 & 0 \end{pmatrix}.$$ If we change from the basis given by $(1,0,0), (0,1,0),$ and $(0,0,1)$ to the basis given by $(0,0,1), (0,1,0),$ and $(1,0,0)$, then these are the matrices $F_1$ and $F_0$ for the permutations $(e,(2\,3),e)$.

\item The matrices for $(e,(2\,3),(2\,3))$ are
$$F_0 = \begin{pmatrix} 0 & 1 & 0 \\ 1 & 0 & 0 \\ 0 & 1 & 1 \end{pmatrix} \text{ and } F_1 = \begin{pmatrix} 1 & 1 & 0 \\ 0 & 0 & 1 \\ 0 & 1 & 0  \end{pmatrix} .$$
We can consider the map given by exchanging $F_0$ and $F_1$ instead. For this map, multiplying gives $$F_1^{a_k}F_0 = \begin{pmatrix} 1&1&0 \\0&0&1 \\ \frac{a_k}{2} & \frac{a_k+2}{2} & \frac{a_k}{2} \end{pmatrix} \text{ for $a_k$ even}, \begin{pmatrix} 0 & 0 & 1 \\ 1 & 1 & 0 \\ \frac{a_k - 1}{2} & \frac{a_k+1}{2} & \frac{a_k + 1}{2}\end{pmatrix} \text{ for $a_k$ odd}.$$
So applying $F_0$ $a_k$ times followed by $F_1$ gives $$(x_{k+1,1},x_{k+1,2},x_{k+1,3}) = (\frac{a_k}{2}x_{k,3} + x_{k,1}, \frac{a_k+2}{2}x_{k,3} + x_{k,1}, \frac{a_k}{2}x_{k,3} + x_{k,2})$$ for $a_k$ even and $$(x_{k+1,1},x_{k+1,2},x_{k+1,3})=(\frac{a_k-1}{2}x_{k,3} + x_{k,2}, \frac{a_k+1}{2}x_{k,3} + x_{k,2}, \frac{a_k+1}{2}x_{k,3} + x_{k,1})$$ for $a_k$ odd.

It is clear from these formulas that $x_{k,2} \geq x_{k,1}$ for all $k$. So $x_{k+1,3} \geq x_{k+1,1}$ for all $k$ as well. Therefore, $x_{k,3} \geq x_{k,1}$ for all $k$. For $a_{k+1}$ odd, it follows that $x_{k+1,2} \geq x_{k+1,3}$.

We claim that $x_{k,2} \geq x_{k,3}$ and $x_{k,1} + x_{k,3} \geq x_{k,2}$ for each $k$. The proof is by induction, with the base case clear. If the result holds for $k$ and $a_k$ is even,
$$x_{k+1,2} = \frac{a_k+2}{2}x_{k,3}+x_{k,1} \geq \frac{a_k}{2}x_{k,3} + x_{k,2} = x_{k+1,3}$$
because $x_{k,1} + x_{k,3} \geq x_{k,2}.$ Also,
$$x_{k+1,1}+x_{k+1,3} = a_kx_{k,3} + x_{k,1} + x_{k,2} \geq \frac{a_k+2}{2} x_{k,3}+x_{k,1} = x_{k+1,2}$$
because $x_{k,2} \geq x_{k,3}.$
If $a_k$ is odd, 
$$x_{k+1,2} = \frac{a_k+1}{2}x_{k,3} + x_{k,2} \geq \frac{a_k+1}{2}x_{k,3}+ x_{k,1} = x_{k+1,3}$$
because $x_{k,2} \geq x_{k,1}$. Further,
$$x_{k+1,1}+x_{k+1,3} = x_{k,1}+x_{k,2} +a_kx_{k,3}  \ge x_{k,2} +\frac{a_k+1}{2} x_{k,3} = x_{k+1,2}.$$
So the result is proved, and $$x_{k,2} \geq x_{k,3} \geq x_{k,1}.$$

Suppose $a_k \neq 0$. If $a_k$ is even, then
$$1 \leq \frac{x_{k+1,2}}{x_{k+1,1}} = \frac{\frac{a_k+2}{2}x_{k,3} + x_{k,1}}{\frac{a_k}{2}x_{k,3}+x_{k,1}}\leq 2.$$
If $a_k$ is odd,
$$1 \leq \frac{x_{k+1,2}}{x_{k+1,1}} = \frac{\frac{a_k+1}{2} x_{k,3} + x_{k,1}}{\frac{a_k-1}{2}x_{k,3} + x_{k,2}} \leq 2.$$
%
%
So the ratios between the $x_{k+1,i}$ are bounded whenever $a_k \neq 0$, which is the case every period if it is ever the case. For all sequences other than $(\overline{0})$ and $(\overline{1})$, apply Corollary 7.1. These two sequences correspond to the top right and bottom left vertices of the triangle, respectively.

\item The permutation $(e,(2\,3), (1\,3\,2))$ gives matrices
$$F_0 = \begin{pmatrix}0&1&0\\1&0&0\\0&1&1 \end{pmatrix} \text{ and } F_1 = \begin{pmatrix} 0 & 1 & 1 \\ 1 & 0 & 0 \\ 0 & 1 & 0 \end{pmatrix}.$$
The permutation $(e,(1\,3\,2),(2\,3))$ gives matrices
$$F_0 = \begin{pmatrix}0&1&0\\0&0&1\\1&1&0 \end{pmatrix} \text{ and } F_1 = \begin{pmatrix} 1 & 1 & 0 \\ 0 & 0 & 1 \\ 0 & 1 & 0 \end{pmatrix}.$$
The matrices for $(e,(1\,3\,2),(1\,3\,2))$ are 
$$F_0 = \begin{pmatrix}0&1&0\\0&0&1\\1&1&0 \end{pmatrix} \text{ and } F_1 = \begin{pmatrix} 0 & 1 & 1 \\ 1 & 0 & 0 \\ 0 & 1 & 0 \end{pmatrix}.$$
We claim that each of these three maps gives the same partition of the triangle as $(e,(2\,3),(2\,3))$. For the base case, the new vertex of $\triangle(i_0)$ is on the edge between $(1,0,0)$ and $(1,1,1)$ for each of these maps. For the inductive step, $v_{k-1,2}$ is the unique vertex of $\triangle(i_0,...,i_{k-1})$ which was not a vertex of $\triangle(i_0,...,i_{k-2})$. Because the new vertex of $\triangle(i_0,...,i_k)$ is $v_{k-1,1} + v_{k-1,3}$ for either choice of $i_k$, and this is the sum of the two vertices that were not new at the previous step, each map gives the same partition. This reduces the verification for the three maps to the verification for $(e,(2\,3),(2\,3))$.


%
%
%

\item The matrices for $(e,(1\,3\,2),e)$ are 
$$F_0 = \begin{pmatrix} 0 & 1 & 0 \\ 0 & 0 & 1 \\ 1 & 1 & 0  \end{pmatrix} \text{ and } F_1= \begin{pmatrix} 1 & 0 & 1 \\ 0 & 1 & 0 \\ 0 & 0 & 1 \end{pmatrix}.$$
Since $$F_1^{a_k}F_0 = \begin{pmatrix} a_k & a_k + 1 & 0 \\ 0 & 0 & 1 \\ 1 & 1 & 0\end{pmatrix},$$ the recursion relation is 
\begin{align*}(x_{k+1,1}, x_{k+1,2},x_{k+1,3}) &= F_1^{a_k}F_0(x_{k,1},x_{k,2},x_{k,3}) \\&= (a_kx_{k,1} + x_{k,3},(a_k+1)x_{k,1} + x_{k,3},x_{k,2}).\end{align*}
It is easy to see that $x_{k+1,2} \geq x_{k+1,1}$, so $x_{k,2} \geq x_{k,1}$ for all $k$. Therefore, $$x_{k+1,1} + x_{k+1,3}= a_kx_{k,1} + x_{k,2} + x_{k,3} \geq (a_k + 1)x_{k,1} +x_{k, 3} =  x_{k+1,2}$$ as well.

Suppose $a_k \neq 0$. Then
$(x_{k+2,1}, x_{k+2,2}, x_{k+2,3})$ is equal to $$(a_{k+1}a_kx_{k,1} + x_{k,2} + a_{k+1}x_{k.3},(a_{k+1}+1)a_kx_{k,1} + x_{k,2} + (a_{k+1}+1)x_{k,3}, (a_k + 1)x_{k,1} + x_{k,3}).$$
If $a_{k+1} \neq 0$ as well, then comparing each $x_{k+2,i}$ to $x_{k+2,1} + x_{k+2,2} + x_{k+2,3}$ shows that their ratios are bounded above by $2 a_k (a_{k+1}+1)$. Because the sequence is periodic, the $a_k$ are bounded, so this is sufficient. If $a_{k+1} = 0$, these coordinates are
$$(x_{k,2}, a_k x_{k,1} + x_{k,2} + x_{k,3}, (a_k + 1)x_{k,1} + x_{k,3}).$$
So $(x_{k+3,1},x_{k+3,2},x_{k+3,3})$ is
$$(a_{k+2}x_{k,2} + (a_k+1)x_{k,1} + x_{k,3}, (a_k + 1)x_{k,1} + (a_{k+2} + 2)x_{k,2} + x_{k,3}, a_kx_{k,1} + x_{k,2} + x_{k,3}).$$
Comparing these entries to $x_{k,1}+x_{k,2}+x_{k,3}$ shows that the ratios between the entries are bounded by $2(a_k a_{k+2}+2)$ for any $a_{k+2}$. Because the series is periodic, Corollary 7.1 gives the result for all sequences except $(\overline{0})$ and $(\overline{1})$. The two exceptions define a single point and the bottom edge of $\triangle$, respectively.
\item The permutation $(e, (1\,2\,3), (1\,3\,2)$ is given by $$F_0 = \begin{pmatrix} 1 & 0 & 0 \\ 0 & 1 & 0 \\ 1 & 0 & 1\end{pmatrix} \text{ and } F_1= \begin{pmatrix} 0 & 1 & 1 \\ 1 & 0 & 0 \\ 0 & 1 & 0\end{pmatrix} .$$ If we change from the basis given by $(1,0,0), (0,1,0),$ and $(0,0,1)$ to the basis given by $(0,0,1), (0,1,0),$ and $(1,0,0)$, then we obtain the matrices $F_1$ and $F_0$ for the permutations $(e,(1\,3\,2),e)$. So the verification is the same.
\end{itemize}
\end{proof}
\end{Proposition}

\section{Generalizations to Higher Dimensions}\label{generalization}

The method of construction for the 216 TRIP Maps extends to
create families of multidimensional continued fractions in any dimension.
We partition the $n$-dimensional simplex using two
$(n+1)\times (n+1)$ matrices and define a map in terms of these partition
matrices. Let $$\Sigma=\{(x_{1},\dots,x_{n})|0\leq x_{n}\leq\dots\leq x_{1}\leq1\}$$
be the $n$-dimensional simplex.

Define
\[
A_{1}=(a_{ij})\mbox{ where \ensuremath{a_{ij}=\begin{cases}
1 & \mbox{if }i=j\mbox{ or }i=n+1\mbox{ }\mbox{and }j=0\\
0 & \mbox{otherwise}
\end{cases}}}
\]

\[
=\left(\begin{array}{ccccc}
1 & 0 & \dots & 0 & 1\\
0 & \ddots & \ddots & \dots & 0\\
\vdots & \ddots & \ddots & \ddots & \vdots\\
0 & \dots & \ddots & \ddots & 0\\
0 & 0 & \dots & 0 & 1
\end{array}\right)
\]

and
\[
A_{0}=(a_{ij})\mbox{ where \ensuremath{a_{ij}=\begin{cases}
1 & \mbox{if }i+1=j\mbox{, }i=n\mbox{ }\mbox{and }j=0,\mbox{ or }i=j=n+1\\
0 & \mbox{otherwise}
\end{cases}}}
\]

\[
=\left(\begin{array}{ccccc}
0 & \dots & \dots & 0 & 1\\
1 & 0 & \dots & \dots & 0\\
0 & \ddots & \ddots & \vdots & \vdots\\
\vdots & \ddots & \ddots & 0 & 0\\
0 & \dots & 0 & 1 & 1
\end{array}\right).
\]

If $v_1, \ldots, v_n$ are column vectors in $\R^n$, we have 
\begin{eqnarray*}
(v_1, \ldots, v_n)A_0 &=& (v_2,v_3,\ldots, v_n, v_1 + v_n)\\
(v_1, \ldots, v_n)A_1 &=&  (v_1,v_2,\ldots, v_{n-1}, v_1 + v_n)
\end{eqnarray*}

These two matrices  will partition $\Sigma$ into two sub-simplices. The
two sub-simplices are constructed by partitioning $\Sigma$ with the
$(n-1)$-dimensional hyperplane containing the points $(1,1,0,\dots,0),$ $(1,1,1,0,\dots, 0),\dots,$ $(1,1,\dots,1,0)$
and $(2,1,\dots,1).$
As can be seen, these two matrices, when the dimension $n=3$, are  matrices $A_0$ and $A_1$ used to construct TRIP Maps.
%
%

As in Section \ref{intro216}, we introduce permutations to create a family of $(n+1)!^{3}$ multidimensional continued
fraction algorithms on the $n$-dimensional simplex.
\begin{defn}
For every $(\sigma,\tau_{0},\tau_{1})\in S_{n+1}^{3},$ define
\[
M_{0}=\sigma A_{0}\tau_{0}\mbox{ and }M_{1}=\sigma A_{1}\tau_{1}
\]

when $\sigma,$ $\tau_{0},$ and $\tau_{1}$ are column permutation
matrices.
\end{defn}
Now to construct the simplex map in the same we earlier constructed our TRIP maps. Let
\[
B_{n}=(b_{ij})\mbox{ where }b_{ij}=\begin{cases}
1 & \mbox{if }i\geq j\\
0 & \mbox{otherwise}
\end{cases}
,\]

define $\pi_{n}:\mathbb{R}^{n+1}\rightarrow\mathbb{R}^{n}$ where
\[
\pi(x_{1},\dots,x_{n+1})=\left(\frac{x_{2}}{x_{1}},\dots,\frac{x_{n+1}}{x_{1}}\right)
,\]

and define $\Sigma_k$ to be the simplex with vertices $\pi_n(v_1)$, $\pi_n(v_2)$, \dots, $\pi_n(v_n),$ where $$(v_1\ v_2\ \dots
v_n)=B_nM_1^kM_0.$$

\begin{defn}
The simplex function is given by
\[
S_{(\sigma,\tau_{0},\tau_{1})_{n}}(x_{1},\dots,x_{n})=\pi_{n}((1\ x_{1}\ x_{2}\dots x_{n})(B_{n}M_{0}^{-1}M_{1}^{-k}B_{n}^{-1})^{T})
\]

when $(x_1,\dots,x_n)\in\Sigma_k.$ The simplex sequence of an element $x_{1},\dots,x_{n}$ of $\Sigma$ is the sequence $\{a_i\}$ such that
\[
[S_{(\sigma,\tau_{0},\tau_{1})_{n}}]^i(x_{1},\dots,x_{n}) \in \Sigma_{a_i}.
\]
\end{defn}

\begin{ex}
The function $S_{((4\ 5),e,e)_{5}}(x_{1},x_{2},x_{3},x_{4})$ is given by
\[
\left(\frac{x_{2}}{x_{1}},\frac{x_{3}}{x_{1}},\frac{(-1)^{-k}\left(-2x_{4}+x_{3}+(-1)^{k}x_{3}\right)}{2x_{1}},-\frac{-4-2x_{4}+4x_{1}+(-1)^{k}(2x_{4}-x_{3})+x_{3}+2kx_{3}}{4x_{1}}\right).
\]

\end{ex}

There is a simple general formula for $S_{(e,e,e)_{n}}$, the $n$-dimensional analogue
of the triangle map.
\begin{thm}
\[
S_{(e,e,e)_{n}}=\left(\frac{x_{2}}{x_{1}},\frac{x_{3}}{x_{1}},\dots.\frac{x_{n-1}}{x_{1}},\frac{1-x-kx_{n-1}}{x_{1}}\right)
\]

where $k=\left\lfloor \frac{1-x_{1}}{x_{n-1}}\right\rfloor .$\end{thm}
\begin{proof}
Consider the matrices
\[
M_{0}=A_{0}=\left(\begin{array}{ccccc}
0 & \dots & \dots & 0 & 1\\
1 & 0 & \dots & \dots & 0\\
0 & \ddots & \ddots & \vdots & \vdots\\
\vdots & \ddots & \ddots & 0 & 0\\
0 & \dots & 0 & 1 & 1
\end{array}\right)
\]

and
\[
M_{1}=A_{1}=\left(\begin{array}{ccccc}
1 & 0 & \dots & 0 & 1\\
0 & \ddots & \ddots & \dots & 0\\
\vdots & \ddots & \ddots & \ddots & \vdots\\
0 & \dots & \ddots & \ddots & 0\\
0 & 0 & \dots & 0 & 1
\end{array}\right).
\]

Note that
\[
M_{1}^{k}=\left(\begin{array}{ccccc}
1 & 0 & \dots & 0 & k\\
0 & \ddots & \ddots & \dots & 0\\
\vdots & \ddots & \ddots & \ddots & \vdots\\
0 & \dots & \ddots & \ddots & 0\\
0 & 0 & \dots & 0 & 1
\end{array}\right),
\]

so
\[
M_{0}^{-1}M_{1}^{-k}=\left(\begin{array}{cccc}
0 & \dots & k & 1+k\\
1 & 0 & \dots & \vdots\\
\vdots & \ddots & \dots & \vdots\\
0 & \dots & 1 & 1
\end{array}\right)^{-1}=\left(\begin{array}{cccc}
0 & 1 & \dots & 0\\
\vdots & 0 & 1 & \vdots\\
-1 & \vdots & \vdots & 1+k\\
1 & 0 & \dots & -k
\end{array}\right).
\]

Thus,
\[
(B_{n}M_{0}^{-1}M_{1}^{-k}B_{n}^{-1})^{T}=\left(\begin{array}{cccc}
0 & \dots & \dots & 1\\
1 & 0 & \dots & -1\\
\vdots & \ddots & 0 & \vdots\\
0 & \dots & 1 & -k
\end{array}\right),
\]

and
\[
S_{(e,e,e)_{n}}=\left(\frac{x_{2}}{x_{1}},\frac{x_{3}}{x_{1}},\dots.\frac{x_{n-1}}{x_{1}},\frac{1-x-kx_{n-1}}{x_{1}}\right).
\]

\end{proof}
We can also prove that a certain class of algebraic numbers has a
purely periodic sequence of period one under these maps.
\begin{thm}
The points of the form $(\alpha,\dots,\alpha^{n-1})$, where $\alpha^{n}+k\alpha^{n-1}+\alpha-1=0,$ are purely periodic of period one under \textup{\emph{$S_{(e,e,e)_{n}}$.
Specifically, $(\alpha,\dots,\alpha^{n-1})$
has}}\textup{ }\textup{\emph{simplex sequence $(k,k,\dots).$}}
\begin{proof}
\[
S_{(e,e,e)_{n}}(\alpha,\dots,\alpha^{n-1})=\left(\frac{\alpha^{2}}{\alpha},\dots,\frac{1-\alpha-k\alpha^{n-1}}{\alpha}\right)
\]
\[
=\left(\frac{\alpha^{2}}{\alpha},\frac{\alpha^{3}}{\alpha},\dots,\frac{\alpha^{n}}{\alpha}\right)
\]
\[
=(\alpha,\dots,\alpha^{n-1}). \qedhere
\]
\end{proof}
\end{thm}

Unfortunately, the  algorithms produced by this method  are not necessarily unique, as can bee seen in the following.

\begin{Proposition}
The algorithms on the tetrahedron generated by the permutations $(e,e,e)$
and $((2\ 3),(1\ 2),(2\ 3))$ are the same.\end{Proposition}
\begin{proof}
Let $v_1, v_2, v_3, v_4$ be the vertices of our simplex.  Then for $(e,e,e),$ we have 
$$(v_1, v_2, v_3, v_4)A_0 = (v_2,v_3,v_4, v_1+v_4)$$
and 
$$(v_1, v_2, v_3, v_4)A_1 = (v_1,v_2,v_3, v_1+v_4).$$
For $((2\ 3),(1\ 2),(2\ 3)),$ set $F_{0}=   (2\ 3)A_{0}(1\ 2)$ and $F_{1}=  (2\ 3)A_{1}(2\ 3)$.  We have 
\begin{eqnarray*}
(v_1, v_2, v_3, v_4)F_0  &=&  (v_1, v_2, v_3, v_4) (2\ 3)A_{0}(1\ 2) \\
&=& (v_1, v_3, v_2, v_4) A_{0}(1\ 2) \\
&=& (v_3, v_2, v_4, v_1+ v_4) (1\ 2) \\
&=& (v_2, v_3, v_4, v_1+v_4) \\
&=& (v_1, v_2, v_3, v_4)A_0  \\
(v_1, v_2, v_3, v_4)F_1  &=&  (v_1, v_2, v_3, v_4)  (2\ 3)A_{1}(2\ 3)  \\
&=& (v_1, v_3, v_2, v_4)  A_{1}(2\ 3)  \\
&=&  (v_1, v_3, v_2, v_1+ v_4) (2\ 3) \\
&=& (v_1, v_2, v_3, v_1+ v_4)\\
&=& (v_1, v_2, v_3, v_4)A_1.
\end{eqnarray*}

Thus the two algorithms are the same.
\end{proof}

By methods similar to the one used in the previous proposition we can derive an expression for the total number of unique multidimensional continued fraction algorithms that our method generates in each dimension.

\begin{lemma}
Let $\sigma\in S_{n}$ leave the first and last columns of a given
$n\times n$ matrix fixed (and hence $\sigma$ sends $1$ to $1$ and $n$ to $n$). . Then there exist $\tau_{0},\tau_{1}\in S_{n}$
such that $\sigma A_{0}\tau_{0}=A_{0}$ and $\sigma A_{1}\tau_{1}=A_{1}.$\end{lemma}
\begin{proof}
We know that $(v_1, \ldots , v_{n-1}, v_n)A_1= (v_1, \ldots , v_{n-1}, v_1+ v_n)$ and hence that $A_1$ only effects the last term.  Thus
by inspection we see for any vectors $v_1, \ldots , v_n$ that 
 $$(v_1, \ldots , v_n)A_1 = (v_1, \ldots , v_n)\sigma A_1\sigma^{-1}.$$
 
 Since we are assuming that $\sigma$ sends $1$ to $1$ and $n$ to $n$, let $\overline{\sigma}\in S_n$ be the permutation that, for $1\leq i \leq n-1$,  sends  $i$ to $j$ if $\sigma$ sends $i+1$ to $j+1$.  Then by inspection we have
 $$(v_1, \ldots , v_n)A_0 = (v_1, \ldots , v_n)\sigma A_0 \overline{\sigma}^{-1}.$$










\end{proof}
\begin{thm}
The number of unique algorithms on the $n$-dimensional simplex is
at most $(n+1)!^{3}\frac{1}{(n-1)!}.$\end{thm}
\begin{proof}
Let $(\sigma,\tau_{0},\tau_{1})\in S_{n+1}^{3}.$ Let $\sigma'\in S_{n+1}$
leave both the first and last columns of a given $n\times n$ matrix
fixed. Note that there are $(n-1)!$ such matrices. Then, by the lemma,
there exist $\tau_{0}'$ and $\tau_{1}'$ such that $\sigma'A_{0}\tau_{0}'=A_{0}$
and $\sigma'A_{1}\tau_{1}'=A_{1}.$ Hence, $(\sigma,\tau_{0},\tau_{1})$
and $(\sigma\sigma',\tau_{0}'\tau_{0},\tau_{1}'\tau_{1})$ generate
the same multidimensional continued fraction in $(n-1)!-1$ cases
since these triples are not equal unless $\sigma'=e$. Therefore,
because $|S_{n+1}^{3}|=(n+1)!^{3},$ the number of unique algorithms
on the $n$-dimensional simplex is at most $(n+1)!^{3}\frac{1}{(n-1)!}.$ \end{proof}

\section{Conclusion}

Motivated by the Hermite problem, we have defined a family of Multidimensional Continued Fractions.  These 216 TRIP Maps include the triangle map and the M\"{o}nkemeyer Map. For each, we have a natural way of assigning a TRIP sequence to a pair of numbers -- the analog of the continued fraction expansion of a number.

By considering combinations of the TRIP Maps, we can describe many previously studied Multidimensional Continued Fractions algorithms.  Thus, we have a common language, mostly consisting of linear algebra, to describe many of the attempts to solve the Hermite problem.  The hope is that this language will facilitate further study of general properties of these algorithms.  In section \ref{uniqueness} of this paper, we have begun this process by characterizing an important property of the TRIP Maps: when their corresponding triangle sequences specify unique points. In \cite{SMALL-Uniqueness}, a countable family of TRIP maps is constructed under which infinitely many pairs in every cubic number field have a purely periodic triangle sequence.

In addition to the direct applications to the Hermite problem, we have worked on several projects concerning the generalization of certain properties of continued fractions.  In \cite{SMALL-TRIP-Stern}, we generalize the connection between continued fractions and the Stern diatomic sequence.  Specifically, we construct an analogous sequence for each of the 216 TRIP Maps, and then explore the structure therein.

In \cite{SMALL-TRIP-Pell} , we generalize the connection between continued fractions and the Pell equation.  Recall that continued fractions provide solutions to the standard Pell equation, and give insight into the units of a related number field.  Given any of our TRIP Maps, we give a method for constructing a ``Pell analog" equation.  Then we show that these Pell analogues share many of the characteristics of the original Pell equation, including providing solutions and producing units in a number field.

The family of 216 multidimensional continued fractions that presented in this paper suggests a variety of questions. Does one of these maps or some composition of these maps solve the Hermite problem? If not, what classes of cubic irrationals will have a periodic triangle sequence under various compositions of maps? As shown in  \cite{SchweigerF08} , the original triangle map is  ergodic. It is well known that the M\"{o}nkemeyer map is ergodic \cite{SchweigerF00}.  What are the dynamics of the other maps in the family of TRIP Maps?  Jensen \cite{Jensen12} has shown that some of the TRIP maps are indeed ergodic.  What about the rest?  Finally, ordinary continued fraction expansions exist for many real-valued functions. Can we find the triangle tree sequences of interesting functions from the triangle $\Delta$ to itself?

\end{document}